\documentclass[11pt, twoside, leqno]{smfart}

\usepackage{amssymb}

\usepackage{hyperref}

\newtheorem{thm}{\indent Th\' eor\`eme}[section]
\newtheorem{cor}[thm]{\indent Corollaire}
\newtheorem{lem}[thm]{\indent Lemme}
\newtheorem{prop}[thm]{\indent Proposition}

\theoremstyle{definition}

\newtheorem{rem}[thm]{Remarque}

\numberwithin{equation}{section}

\def\boxit#1#2{\setbox1=\hbox{\kern#1{#2}\kern#1}%
\dimen1=\ht1 \advance\dimen1 by #1 \dimen2=\dp1 \advance\dimen2 by #1
\setbox1=\hbox{\vrule height\dimen1 depth\dimen2\box1\vrule}%
\setbox1=\vbox{\hrule\box1\hrule}%
\advance\dimen1 by .4pt \ht1=\dimen1
\advance\dimen2 by .4pt \dp1=\dimen2 \box1\relax}

\def\hfl#1#2{\smash{\mathop{\hbox to 12 mm{\rightarrowfill}}
\limits^{\scriptstyle#1}_{\scriptstyle#2}}}

 \def\hlfl#1#2{\smash{\mathop{\hbox to 12 mm{\leftarrowfill}}
\limits^{\scriptstyle#1}_{\scriptstyle#2}}}

\def\phfl#1#2{\smash{\mathop{\hbox to 8 mm{\rightarrowfill}}
\limits^{\scriptstyle#1}_{\scriptstyle#2}}}

 \def\phlfl#1#2{\smash{\mathop{\hbox to 8 mm{\leftarrowfill}}
\limits^{\scriptstyle#1}_{\scriptstyle#2}}}

\def\atop#1#2{
\genfrac{}{}{0pt} {} 
{#1} 
{#2}}

\def\cqfd{\unskip\kern 6pt\penalty 500
\raise -2pt\hbox{\vrule\vbox to 10pt{\hrule width 4pt
\vfill\hrule}\vrule}\par}

\def\house#1{\setbox1=\hbox{$\,#1\,$}%
\dimen1=\ht1 \advance\dimen1 by 2pt \dimen2=\dp1 \advance\dimen2 by 2pt
\setbox1=\hbox{\vrule height\dimen1 depth\dimen2\box1\vrule}%
\setbox1=\vbox{\hrule\box1}%
\advance\dimen1 by .4pt \ht1=\dimen1
\advance\dimen2 by .4pt \dp1=\dimen2 \box1\relax}

\def\virgule{\raise 2pt \hbox{$,$}}

\def\Q{\mathbb{Q}}
\def\R{\mathbb{R}}

\def\Z{\mathbb{Z}}

\def\rmd{{\mathrm d}}
\def\rmm{{\mathrm m}}
\def\rmL{{\mathrm L}}

\def\calO{{\mathcal{O}}}
 
 \def\GL{{\mathrm{GL}}}
\def\calA{\mathcal{A}}

\def\og{\leavevmode\raise.3ex\hbox{$\scriptscriptstyle
\langle\!\langle$}}

\def\fg{\leavevmode\raise.3ex\hbox{$\scriptscriptstyle
\,\rangle\!\rangle$}}

\begin{document}

\begin{center}
\Large\bf 
Sur la repr\'esentation des entiers 
\\
par les formes cyclotomiques de grand degr\' e  

\bigskip

{\sc \small \'Etienne Fouvry \&\ Michel Waldschmidt}

\author{{\sc \' Etienne Fouvry}\\
\rm
Institut de Math\'ematique d'Orsay \\
 Universit\'{e} Paris--Sud \\
CNRS, Universit\'{e} Paris--Saclay\\
 F--91405 ORSAY, \sc France\\
 \rm
 E-mail: \url{etienne.fouvry@u-psud.fr}
}

\author{
 {\sc Michel Waldschmidt }\\
 \rm
 Sorbonne Universit\'e \\ 
 Institut Math\'ematique de Jussieu IMJ-PRG \\
 F--75005 Paris, \sc France \\
 \rm
 E-mail: 
 \url{michel.waldschmidt@imj-prg.fr}}

\end{center}
Pour chaque entier $d\ge 4$, nous \' etudions la suite des entiers positifs repr\' esent\' es par une des formes binaires cyclotomiques $\Phi_n(X,Y)$ pour les $n$ positifs tels que $\varphi(n)\ge d$. Le cas $d=2$ a 
\' et\' e 
 \' etudi\' e dans notre pr\' ec\'edent texte \cite{FLW}. Notre d\'emonstration repose sur une variante d'un \' enonc\' e de \cite{SX} concernant les valeurs communes prises par deux formes binaires de m\^eme degr\' e et de discriminants non nuls. 

Toutes les constantes sont effectivement calculables.

\section{Introduction}\label{section:introd}
Rappelons (\cite{FLW}) que la suite $(\Phi_n (X,Y))_{n\geq 1}$ des {\it formes cyclotomiques } est d\' efinie par la formule de r\' ecurrence
$$
X^n-Y^n =\prod_{k \mid n} \Phi_k (X,Y).
$$
Le polyn\^ome $\Phi_n (X,Y)$ est homog\`ene de degr\' e $\varphi (n)$ ($\varphi $ fonction indicatrice d'Euler) et il est reli\' e au polyn\^ome cyclotomique $\phi_n (t) \in \mathbb Z [t]$ par la formule

$$
\Phi_n (X,Y) =Y^{\varphi (n)} \phi_n (X/Y).
$$
Puisque, pour $n\geq 3$, le polyn\^ome $\phi_n (t)$ n'a aucun z\' ero r\' eel, on a donc l'in\' egalit\' e 
$$
\Phi_n (x,y) \gg \max (\vert x\vert ^{\varphi (n)},\vert y \vert^{\varphi (n)}),
$$
uniform\' ement sur $x$ et $y$ r\' eels. 

Pour $N\geq 2$ et $d$ entier pair, on d\' esigne par $\mathcal A_d (N)$ le cardinal de l'ensemble des entiers $1\leq m \leq N$ tels qu'il existe un entier $n$ et des entiers $(x,y)$ v\' erifiant les trois conditions
\begin{equation}\label{3cond}
\begin{cases}
\varphi (n) \geq d,\\
\Phi_n (x,y) =m,\\
\max(\vert x\vert, \vert y \vert ) \geq 2.
\end{cases}
\end{equation}
On s'int\'eresse au comportement asymptotique de $\mathcal A_d (N)$ pour $d$ fix\' e et $N$ tendant vers l'infini. Il est alors sage d'introduire la derni\`ere condition de \eqref{3cond} puisque pour tout $p$ premier on a $\Phi_p (1,1)=p$ et le cardinal des $p\leq N$
masquerait le terme principal de l'estimation qui sera donn\'ee en \eqref{central}. Par convention, nous r\' eservons la lettre $p$
aux nombres premiers. Enfin, si $n$ est un entier impair,
on a l'\' egalit\' e
\begin{equation*} 
\Phi _{2n} (X,Y) = \Phi_n (X,-Y).
\end{equation*}
On peut ainsi ajouter, aux conditions de \eqref{3cond}, la condition de congruence
\begin{equation}\label{nnequiv}
n\not\equiv 2 \bmod 4,
\end{equation}
sans modifier l'\' etude de $\mathcal A_d (N)$.

 Appelons {\it totient} toute valeur prise par la fonction $\varphi$. La suite 
croissante des totients est ainsi
$$
\mathfrak T := \{1,\, 2,\, 4,\, 6,\, 8,\, 10, \, 12,\, 16, \, 18, \, 20, \, 22,\, 24,\, 28,\, 30,\, \dots \}.
$$
Cette suite contient la suite des $p-1$ mais reste myst\' erieuse \`a de nombreux points de vue (on se reportera avec profit aux articles de Ford \cite{Ford1} et \cite{Ford2} traitant, entre autres choses,
de la fonction de comptage de la suite $\mathfrak T$ et du nombre de solutions \`a l'\' equation $ \varphi (n)=d$
). Il est naturel de restreindre l'\' etude de $\mathcal A_d (N)$ au cas o\`u $d$ est un totient pair.
L'\'etude de $\mathcal A_2 (N)$ a \'et\' e trait\'ee dans \cite[Th\' eor\`eme 1.3]{FLW} o\`u il est prouv\' e qu'il existe deux constantes 
 $C_2= 1, 403132\dots$ et $C'_2=0,302316\dots $ telles que, uniform\' ement pour $N \geq 2$, on a l'\' egalit\' e

 \begin{equation}\label{A2=}
 \calA_2(N)=C_2 \frac{N}{(\log N)^{\frac 12}}- C'_2\frac{N}{(\log N)^{\frac 34 } } + O\left(
 \frac{N}{(\log N)^{\frac 32}}
 \right).
\end{equation}
 Les constantes $C_2 $ et $C'_2 $ se d\' efinissent au moyen des valeurs, au point $s=1$, de certaines fonctions de Dirichlet $L(s, \chi)$ o\`u $\chi$ est le caract\`ere de Kronecker attach\' e aux corps quadratiques de discriminants $-4$, $-3$ et $12$.

 Pour $d$ totient $\geq 4$, les outils de \cite{FLW} conduisent \`a la majoration 
 \begin{equation}\label{Ad<<}
 \mathcal A_d (N) = O \bigl( N^\frac 2d (\log N)^{1,161}\bigr),
 \end{equation}
 (voir corollaire \ref{Corollaire:FLW} et sa preuve ci--dessous).
 
 Par rapport \`a \cite{FLW}, le pr\' esent travail innove en injectant r\'esultats et m\' ethodes de \cite{SX} que nous d\' ecrirons au \S \ref{StewartYaonew}. Contentons--nous pour l'instant de donner notre r\' esultat principal qui am\' eliore notablement \eqref{Ad<<}. Pour son \' enonc\' e, nous introduisons les notations suivantes :
 
 \noindent $\bullet$ si $d$ est un totient, on note $d^\dag$ le successeur imm\' ediat de $d$ dans la suite $\mathfrak T$.

 \noindent $\bullet $ pour $d $ entier $\geq 3$, on pose 
\begin{equation}\label{defeta}
\eta_d = 
\begin{cases}
&
 2/9 + 73/( 108 \sqrt 3) \text{ si } d=3, \\
&
(1/2 +9/(4 \sqrt d))/d\text{ si } 4\leq d \leq 20,\\ 
&
1/d \text{ pour } d\geq 21. 
\end{cases}
\end{equation}
 On prouvera donc le 
 \begin{thm}\label{thmcentral} Soit $d\geq 4$ un totient. Alors, il existe une constante $C_d >0$, telle 
 que, pour tout $\varepsilon >0$ et uniform\' ement pour $N\geq 2$, on a l'\' egalit\' e 
 \begin{equation}\label{central}
 \mathcal A_d (N)
 = C_d N^\frac{2}{d} + O (N^\frac{2}{d^\dag}) + O_\varepsilon( N^{\eta_d +\varepsilon}).
 \end{equation}
 
 \end{thm}
 
 \begin{rem} 
La formule \eqref{central} est d'autant plus pr\' ecise que $d^\dag -d$ est grand. Ainsi, dans le cas particulier o\`u $ d\geq 6$, la minoration triviale
 $$
 d^\dag \geq d+2,
 $$
r\' eduit la formule \eqref{central} en sa forme plus grossi\`ere
 $$
 \mathcal A_d (N)= C_d N^\frac 2d
 +O( N^\frac 2{d+2}).
 $$
 \end{rem}
 
 \begin{rem} Le th\' eor\`eme \ref{thmcentral} suppose $d\ge 4$. La formule \eqref{A2=} correspond donc au cas $d=2$. Mais, par la pr\'esence au d\' enominateur du facteur $(\log N)^\frac 12$, elle diff\`ere notablement de \eqref{central}. Cette diff\'erence s'explique comme suit. Il y a trois formes cyclotomiques de degr\' e $2$. Ce sont les trois formes quadratiques binaires
 \begin{equation}\label{Phi3Phi4Phi6}
 \Phi_3 (X,Y) = X^2+XY +Y^2, \, \Phi_4 (X,Y) =X^2+Y^2, \text{ et } \Phi_6 (X,Y) = X^2-XY+Y^2.
 \end{equation}
 Puisque $\Phi_6 (X,-Y) = \Phi_3 (X,Y)$ les formes $\Phi_6$ et $\Phi_3$ repr\' esentent les m\^emes entiers. Mais les formes $\Phi_3$ et $\Phi_4$ \`a la diff\'erence 
 des formes cyclotomiques de degr\' e au moins $4$, ont un nombre infini d'automorphismes comme d\'efinis au \S \ref{autoforcyclo}. Par exemple on a ${\rm Aut } \Phi_4 = {\rm O}(2, \mathbb Q)$ (le groupe des matrices orthogonales $2\times 2$ \`a coefficients dans $\Q$).
 \end{rem}
 L'objet des th\' eor\`emes \ref{thmcentral2} et \ref{thmcentral3} est de compl\' eter la formule \eqref{central}. Nous pr\' ecisons d'abord la constante $C_d$.
 
 \begin{thm} \label{thmcentral2}Soit $d\geq 4$ un totient. La constante $C_d$ de la formule \eqref{central} v\' erifie l'\' egalit\' e
\begin{equation}\label{defCd}
 C_d= \sum_{\substack{n\not\equiv 2 \bmod 4 \\ \varphi (n) = d}} \, w_n\,A _{\Phi_n} 
\end{equation}
 o\`u 
\begin{equation}\label{Equation:wn}
 w_n: =
 \begin{cases} \frac 14 \text{ si } 4\nmid n,\\
 \frac 18 \text{ si } 4\mid n,
 \end{cases}
\end{equation}
 et
 $$
 A_{\Phi_n}= \iint_{\Phi_n(x,y) \leq 1}{\rm d}x{\rm d}y.
 $$
 \end{thm}

Voici deux exemples dans lesquels la formule \eqref{defCd} donnant la valeur de la constante $C_d$ se simplifie. 

\begin{enumerate}
\item
Soit $p\ge 5$ un nombre premier de Sophie Germain, c'est-\`a-dire tel que le nombre $\ell=2p+1$ soit premier. 
Alors $\ell$ est l'unique entier $\not\equiv 2 \bmod 4$ tel que $\varphi (\ell) =2p$ et on a l'\'egalit\'e 
$$
C_{2p} = \frac 1 4 A_{\Phi_\ell}.
$$ 
On conjecture qu'il y a une infinit\'e de nombres premiers de Sophie Germain. 

\item
Supposons que $d\ge 4$ est une puissance de $2$, disons $d=2^k$. 

On d\'esigne par $\mathcal M$ l'ensemble des nombres entiers $m\ge 1$ dont le d\'eveloppement binaire $m=2^{a_1}+2^{a_2}+\cdots+2^{a_r}$ est tel que chacun des nombres $F_{a_i}=2^{2^{a_i}}+1$ est premier (nombre premier de Fermat). L'ensemble $\mathcal M$ contient les entiers $1,2,3,\dots,31$; on ne conna\^\i t pas d'autre \'el\'ement de $\mathcal M$. Pour chaque $m\in \mathcal M$ v\'erifiant $m\le k$, on d\'efinit 
$$
\ell_k(m)=2^{k-m+1}F_{a_1}F_{a_2}\cdots F_{a_r},
$$
de sorte que $\varphi(\ell_k(m))=2^k$. Les entiers $n\not \equiv 2 \bmod 4$ tels que $\varphi(n)=d$ sont d'une part $n=2d$, qui est multiple de $4$, d'autre part les $\ell_k(m)$ avec $m< k$, qui sont aussi multiples de $4$, et enfin, si $k\in \mathcal M$, $\ell_k(k)/2$ qui est impair, avec $A_{\Phi_{\ell_k(k)}}=A_{\Phi_{\ell_k(k)/2}}$. Alors 
$$
C_d=
\begin{cases}
\displaystyle
\frac 1 8 A_{\Phi_{2d}} + \frac 1 8 \sum_{\atop{m\in \mathcal M}{m< k}}A_{\Phi_{\ell_k(m)}} &\hbox{si $k\not\in \mathcal M$},
\\
\displaystyle
\frac 1 8 A_{\Phi_{2d}} + \frac 1 8 \sum_{\atop{m\in \mathcal M}{m< k}}A_{\Phi_{\ell_k(m)}} +\frac 1 4
A_{\Phi_{\ell_k(k)}}
 &\hbox{si $k\in \mathcal M$,}
\end{cases}
$$
avec
$$
A_{\Phi_{2d}}=\int_{-\infty}^\infty \frac {\rmd t}{(1+t^d)^{2/d}}= \frac{2}{d}\frac{\Gamma(1/d)^2}{\Gamma(2/d)}
$$
(cf. \S\ref{SS:regionfondamentale} et \cite[Corollaire 1.3 et \S~5]{SX}). 
\end{enumerate}

 \bigskip
Nous montrerons que l'on a 
 $$
 \lim_{n\rightarrow \infty} A_{\Phi_n} =4.
 $$
C'est une cons\'equence de l'\'enonc\'e plus pr\'ecis suivant, concernant le {\it domaine fondamental cyclotomique } $\calO_n$ 
d\'efini par 
$$
\calO_n= \{(x,y)\in\R^2\; \mid \; \Phi_n(x,y)\le 1\}.
 $$

\begin{thm}\label{Thm:ovalecyclotomique}
Soit $\varepsilon>0$. Il existe $n_0=n_0(\varepsilon)$ tel que, pour $n\ge n_0$, le domaine fondamental cyclotomique $\calO_n$ d'indice $n$ contient le carr\'e centr\'e en $O$ de c\^ot\'e $2-n^{-1+\varepsilon}$ et est contenu dans le carr\'e centr\'e en $O$ de c\^ot\'e $2+n^{-1+\varepsilon}$. 
\end{thm}

 Enfin nous discutons de l'optimalit\' e de la formule \eqref{central}
 \begin{thm}\label{thmcentral3}
 On adopte les notations du th\' eor\`eme \ref{thmcentral}. Soit $d\geq 4$ un entier tel que $d$ et $d+2$ soient des totients. Il existe une constante positive $v_d>0$ telle que pour $N$ suffisamment grand, on ait l'in\'egalit\' e 
$$
 \calA_d(N) \geq C_d N^{\frac 2d} 
+ v_d N^\frac{2}{d+2}.
$$
\end{thm}
\begin{rem} Il est naturel de conjecturer qu'il y a une infinit\' e de $d$ tels que $d+2$ soit aussi un totient : c'est une cons\' equence de la conjecture des nombres premiers jumeaux. Enfin, on peut tout \`a fait envisager
des \' enonc\' es analogues sous l'hypoth\`ese $d^\dag =d+\nu$ o\`u $\nu$ est un entier pair fix\' e. Cette extension n\' ecessiterait une adaptation des propri\' et\' es de confinement d\' ecrites aux \S \ref{confinemodp}, \S \ref{confinemod9} et \S \ref{confinemod4}.
\end{rem}
 
\section{Valeurs prises par une forme binaire}\label{StewartYaonew}
L'objet de cette section est de d\' ecrire pr\' ecis\' ement le r\' esultat de Stewart et Xiao \cite{SX}. Ce r\' esultat d\'eja mentionn\' e plus haut est \`a la base de notre travail.
 Dans toute cette section $F= F(X,Y)$ est un polyn\^ome homog\`ene de $\mathbb Z [X,Y]$ de degr\' e $d\geq 3$. On dit alors que $F$ est une {\it forme binaire de degr\' e $d$}. 
Un entier $m$ est dit {\it repr\' esent\'e} par une forme binaire $F$ s'il existe $(x,y)\in\Z^2$ tel que $F(x,y)=m$. On d\' esigne par $R_F (N)$ le cardinal de l'ensemble des entiers $m$ repr\' esent\' es
par $F$ et v\' erifiant $0\leq \vert m \vert \leq N$. On appelle {\it automorphisme } de $F$ toute matrice $U$ de ${\rm Gl}(2, \mathbb Q)$
$$
U=\begin{pmatrix} u_1 & u_2\\
u_3 & u_4
\end{pmatrix},
$$
telle que 
\begin{equation}\label{defauto}
F (X,Y) =F (u_1X +u_2Y, u_3X+u_4 Y).
\end{equation}
Muni de la multiplication des matrices, l'ensemble des automorphismes de $F$ forme un sous--groupe fini de ${\rm Gl } (2, \mathbb Q)$, not\'e ${\rm Aut } F$. Il existe alors une ensemble $\mathcal G$ de dix sous--groupes finis de ${\rm Gl}(2,\mathbb Z)$ tel que que pour toute forme binaire $F$ de degr\' e $d \geq 3$, il existe $T \in {\rm Gl}(2, \mathbb Q)$ et un unique $G \in \mathcal G$ v\'erifiant 
$$
{\rm Aut} F = T G T^{-1}. 
$$
L'ensemble $\mathcal G$ contient en particulier les deux groupes suivants
\begin{enumerate}
\item $\mathbb D_2$, groupe di\'edral \`a quatre \' el\' ements, engendr\' e par 
$$
\begin{pmatrix}
 0&1\\ 1 & 0 
\end{pmatrix}
\text{ et } 
\begin{pmatrix}
-1 & 0\\ 0 & -1
\end{pmatrix},
$$

\item $\mathbb D_4$, groupe di\'edral \`a huit \' el\' ements, engendr\' e par 
$$
\begin{pmatrix}
 0&1\\ 1 & 0 
\end{pmatrix}
\text{ et } 
\begin{pmatrix}
0 & 1\\ -1 & 0
\end{pmatrix}.
$$
\end{enumerate} 
\`A partir de la d\' ecomposition pr\' ec\'edente, Stewart et Xiao construisent un nombre rationnel $W_F$ (voir \cite[Theorem 1.2]{SX}) dont la d\' efinition, de nature alg\'ebrique, est longue puisqu'elle envisage les dix possibilit\' es pour $G$. Disons, pour m\'emoire, que $W_F$ tient compte des d\' eterminants des r\' eseaux de $\mathbb Z^2$ dont l'image, par certains sous--groupes de ${\rm Aut } F$, est incluse dans $\mathbb Z^2$. En vue des applications,
nous nous restreignons \`a deux cas particuliers
\begin{enumerate}
\item Cas o\`u $G=\mathbb D_2$. Soit $\Lambda$ le sous--r\'eseau des \' el\' ements $(v,w)\in \mathbb Z^2$ tels que, pour tout $A \in {\rm Aut } F$ on ait $A \begin{pmatrix} v \\ w\end{pmatrix}\in \mathbb Z^2.$ On pose alors
\begin{equation}\label{defWF1}
W_F= \frac 12 \Bigl( 1 -\frac 1{2\vert \det (\Lambda)\vert}\Bigr).
\end{equation}
\item Cas o\`u $G=\mathbb D_4$. D'abord $\Lambda$ est d\' efini comme pr\' ec\' edemment. Le groupe ${\rm Aut}F $ poss\`ede exactement trois sous--groupes de cardinal $4$, not\' es $G_1$, $G_2$ et $G_3$. On d\' esigne par
$\Lambda_i$ ($1\leq i \leq 3$) le sous--r\'eseau des \' el\' ements $(v,w)\in \mathbb Z^2$ tels que, pour tout $A \in G_i$ on ait $A \begin{pmatrix} v \\ w\end{pmatrix}\in \mathbb Z^2.$ 
On pose alors 
\begin{equation}\label{defWF2}
W_F =\frac 12 \Bigl( 1 -\frac 1{2 \vert \det (\Lambda_1) \vert} -\frac 1{2 \vert \det (\Lambda_2) \vert} -\frac 1{2 \vert \det (\Lambda_3) \vert} + \frac 3{4 \vert \det (\Lambda) \vert}\Bigr).
\end{equation}

\end{enumerate}
Notons
$$
A_F := \iint_{\vert F (x,y) \vert \leq 1} {\rm d } x {\rm d } y, 
$$
l'aire de la r\'egion fondamentale associ\' ee \`a $F$. 
 Enfin, pour $d\geq 4$ entier pair, nous introduisons la constante $\beta^*_d$ d\'efinie par 
 \begin{equation}\label{defbeta*}
\beta^*_d = \begin{cases}
 3/(d\sqrt d) \text{ pour } d=4,\, 6, \, 8
 \\ 1/d \text{ pour } d\geq 10.
 \end{cases}
 \end{equation}
Nous pouvons maintenant \' enoncer le r\' esultat fondamental de \cite[Theorem 1.1]{SX}. 
\begin{thm} \label{thm21}
Pour tout $d\geq 3$ il existe une constante $\beta_d <2/d$ ayant la propri\' et\' e suivante : 
Pour toute forme binaire $F$ de degr\' e $d$, de discriminant non nul, pour tout $\varepsilon >0$, on a, uniform\' ement pour $N \geq 2$ l'\'egalit\' e
\begin{equation}\label{RF(N)=}
R_F (N) = A_F \, W_F N^\frac 2d + O_{F,\varepsilon} (N^{\beta_d +\varepsilon}). 
\end{equation}
Si, dans la formule \eqref{RF(N)=}, on se restreint aux formes binaires $F$ de degr\' e pair $d\geq 4$, de discriminant non nul, sans facteur lin\'eaire r\' eel, 
on peut donner \`a
 $\beta_d $ la valeur $\beta_d^*$, d\' efinie en \eqref{defbeta*}.
\end{thm}
\section{Valeurs communes \`a deux formes binaires} Aux formes binaires $F_1=F_1 (X_1, X_2)$ et $F_2= F_2 (X_3, X_4) $ on associe les deux fonctions de comptage suivantes :
\begin{enumerate}
\item
Pour $B\geq 2$, 
$ \mathcal N_{F_1,F_2} (B)$ est 
le cardinal de l'ensemble 
$$
\Bigl\{ (x_1,x_2,x_3, x_4)\in \mathbb Z^4 \, \mid \, \max_{i=1, \, 2,\, 3,\, 4} \vert x_i\vert \leq B, \ F_1 (x_1, x_2) = F_2 (x_3, x_4)\Bigr\}.
 $$
 \item 
 Pour $N \geq 2$, $R_{F_1,F_2} (N)$ est le cardinal de l'ensemble
 $$
 \Bigl\{ n \,\mid \, \vert n\vert \leq N,\, n =F_1(x_1,x_2) = F_2 (x_3, x_4), \text{ pour certains } (x_1,x_2,x_3,x_4) \in \mathbb Z^4
 \Bigr\}.
 $$
 \end{enumerate}
 G\' en\' eralisant la d\' efinition \eqref{defauto}, on dit que $F_1$ et $F_2$ sont {\it isomorphes} s'il existe $U \in {\rm Gl} (2, \mathbb Q)$ tel que
 \begin{equation}\label{isom}
 F_1(X_1, X_2) =F_2 (u_1X_1 +u_2 X_2, u_3 X_1 +u_4 X_2).
 \end{equation}
 En particulier deux formes isomorphes ont m\^eme degr\' e.
Nous prouverons au \S \ref{preuve} le
\begin{thm}\label{thmF1=F2} Soient $F_1$ et $F_2$ deux formes non isomorphes de m\^eme degr\' e $d \geq 3$. On suppose de plus
que les discriminants de $F_1$ et de $F_2$ sont non nuls et qu'au moins une des formes $F_i$ n'est pas divisible
par une forme lin\'eaire non nulle \`a coefficients rationnels. Alors, pour tout $\varepsilon >0$ on a
la majoration
\begin{equation}\label{daisy}
\mathcal N_{F_1,F_2} (B) =
 O( B^{d\eta_d +\varepsilon}),
 \end{equation}
 o\`u $\eta_d$ est d\' efini en \eqref{defeta}.

\end{thm}

\begin{rem} La condition que l'une des formes $F_1$ et $F_2$ ne contient pas de facteur lin\' eaire sur $\mathbb Q$ est importante.
Consid\' erons les deux formes 
$$
F_1 (X_1, X_2) = X_1 (X_1^2+X_2^2)
$$
et
$$
F_2 (X_3, X_4) = X_3 (X_3^2 +2 X_4^2).
$$
Elles ne sont pas isomorphes. L'\' egalit\'e $F_1 (0, x_2) = F_2 (0, x_4)=0$ implique l'in\'egalit\' e $\mathcal N_{F_1,F_2} (B) \gg B^2$, ce qui est sup\'erieur \`a la
partie droite de \eqref{daisy}.
\end{rem} 
Du th\' eor\`eme \ref{thmF1=F2} nous d\' eduirons le
 \begin{cor} \label{F1=F2}Soient $F_1$ et $F_2$ deux formes v\' erifiant les hypoth\`eses du th\' eor\`eme~\ref{thmF1=F2}. On suppose de plus que les deux formes $F_1$ et $F_2$ sont d\'efinies positives.
 Alors pour tout $\varepsilon >0$ on a l'in\' egalit\' e $$R_{F_1, F_2} (N)
 \ll N^{ \eta_d +\varepsilon}.
 $$
 \end{cor}
\begin{proof}[D\' emonstration du corollaire \ref{F1=F2}] 
Les hypoth\`eses impliquent les in\' egalit\' es $$\vert \, F_1 (x_1,x_2) \, \vert \gg \max (\vert x_1^d\vert , \vert x_2^d\vert )\text{ et } \vert \, F_2 (x_3,x_4)\, \vert \gg \max (\vert x_3^d \vert , \vert x_4^d\vert ),$$ uniform\' ement pour $x_i \in\mathbb R$.
Par cons\'equent, si on a 
$$
-N\leq F_1 (x_1,x_2) =F_2 (x_3,x_4) \leq N,
$$
on a alors 
$$ \max_{i=1, \, 2,\, 3,\, 4} \vert x_i\vert \ll N^\frac 1d.$$
Il suffit alors de remplacer $B$ par $O(N^\frac 1d)$ dans la majoration \eqref{daisy}.
\end{proof}

\subsection{Preuve du th\' eor\`eme \ref{thmF1=F2} } \label{preuve}
Puisque les formes $F_1$ et $F_2$ sont de discriminants non nuls, l'hypersurface $\mathbb X$ de $\mathbb P^3 (\mathbb C)$ 
d\' efinie par l'\' equation
\begin{equation}\label{defX}
\mathbb X : F_1(X_1, X_2)- F_2 (X_3, X_4) =0
\end{equation}
est lisse. Nous inspirant de \cite{SX}, nous d\' ecomposons $\mathcal N_{F_1,F_2} (B)$ en 
\begin{equation}\label{N=N1+N2}
\mathcal N_{F_1,F_2} (B) = \mathcal N_{F_1,F_2}^{(1)} (B) +\mathcal N_{F_1,F_2}^{(2)} (B) +1,
\end{equation}
o\`u 

\noindent $\bullet $ $\mathcal N_{F_1,F_2}^{(1)} (B)$ est le nombre de quadruplets non nuls $(x_1, x_2,x_3,x_4)$ de $\mathbb Z^4$, v\' erifiant $ \max \vert x_i\vert \leq B$,
et tels que le point projectif associ\'e $(x_1 : x_2: x_3 : x_4)$ appartienne \`a $\mathbb X$, mais ne se situe pas sur une droite (complexe) contenue dans $\mathbb X$,
\vskip .3cm
\noindent $\bullet $ $\mathcal N_{F_1,F_2}^{(2)} (B)$ est le nombre de quadruplets non nuls $(x_1, x_2,x_3,x_4)$ de $\mathbb Z^4$, v\' erifiant $ \max \vert x_i\vert \leq B$,
et tels que le point projectif associ\'e $(x_1 : x_2: x_3 : x_4)$ appartienne \`a une droite (complexe) contenue dans $\mathbb X$.

Nous prouverons d'abord la 
\begin{prop}\label{N1<<} Pour tout $\varepsilon >0$ et uniform\' ement pour $B\geq 1$, on a l'in\' egalit\' e
$$
 \mathcal N_{F_1,F_2}^{(1)} (B) \ll B^{d\eta_d+ \varepsilon}.
$$
\end{prop}

\noindent
Puis la 
\begin{prop}\label{N2<<}
Uniform\' ement pour $B\geq 1$, on a l'in\' egalit\' e
$$
 \mathcal N_{F_1,F_2}^{(2)} (B) \ll B.
 $$
\end{prop}
En combinant ces deux propositions et la formule \eqref{N=N1+N2} on compl\`ete la preuve du th\' eor\`eme \ref{thmF1=F2}.
\subsubsection{Preuve de la Proposition \ref{N1<<} } Si $\boldsymbol x = (x_1 : x_2 : x_3,:x_4)$ est un point de $\mathbb P^3 (\mathbb Q)$, on d\' esigne
par $h(\boldsymbol x)$ la hauteur de $\boldsymbol x$, c'est--\`a--dire le maximum des $\vert x_i\vert$, si $(x_1, x_2,x_3,x_4)$ est un quadruplet d'entiers premiers entre eux dans leur ensemble, repr\' esentant $\boldsymbol x$. Notons aussi
$
N^{(1)}(\mathbb X, B)$
le cardinal de l'ensemble des $\boldsymbol x$ de $\mathbb P^3 (\mathbb Q)$ appartenant \`a $\mathbb X$ mais non situ\' es sur une droite contenue dans $\mathbb X$
et de hauteur $h(\boldsymbol x)\leq B$. En d\' ecomposant suivant la valeur $\delta$ du pgcd de $(x_1,x_2,x_3,x_4)$ on d\' eduit l'in\'egalit\' e
\begin{equation}\label{sumoverdelta}
 \mathcal N_{F_1,F_2}^{(1)} (B)\leq \sum_{1\leq \delta \leq B} N^{(1)}(\mathbb X, B/\delta).
\end{equation}
Un r\' esultat de Salberger \cite[Theorem 0.1]{Sa}, donne l'in\' egalit\' e
$$
N^{(1)}(\mathbb Y, B)
\ll B^{d \eta_d+ \varepsilon}
$$
valable pour toute surface projective
$\mathbb Y \subset \mathbb P^n(\mathbb Q)$, g\' eom\' etriquement int\`egre 
de degr\' e $d >2$. La surface $\mathbb X$ v\' erifie cette propri\' et\' e. En effet supposons qu'il existe des polyn\^omes $A$ et $B$ de degr\' e $\geq 1$ tels que
$$
F_1(X_1, X_2)-F_2(X_3, X_4) = A( X_1, X_2, X_3, X_4) B (X_1, X_2, X_3, X_4).
$$
Calculant les d\' eriv\'ees partielles par rapport \`a chacun des $X_i$, on voit que tout point $(x_1 : x_2: x_3: x_4) $ de $\mathbb P^3 (\mathbb C)$
tel que
$$
A( x_1, x_2, x_3, x_4)= B( x_1, x_2, x_3, x_4) =0,
$$
est un point singulier de $\mathbb X$. Puisqu'on est dans $\mathbb P^3 (\mathbb C)$, les deux surfaces d'\'equation $A=0$ et $B=0$ ont une intersection non vide.
Ainsi $\mathbb X$ serait singuli\`ere, ce qui contredit la propri\' et\' e de lissit\' e \' enonc\' ee au \S \ref{preuve}.
 
 Il suffit de
 sommer sur $\delta < B$ l'in\' egalit\' e \eqref{sumoverdelta} pour compl\'eter la preuve de la Proposition \ref{N1<<}.
\subsubsection{Droites contenues dans $\mathbb X$ }

Afin de d\'emontrer la Proposition \ref{N2<<}, nous donnons des conditions n\' ecessaires pour qu'une droite projective de $\mathbb P^3 (\mathbb C)$ appartienne \`a $\mathbb X$ en exploitant le fait que dans l'\' equation \eqref{defX} d\'efinissant $\mathbb X$, les paires de variables $(X_1,X_2)$ et $(X_3,X_4)$ sont s\' epar\' ees.

\begin{lem}\label{lemme:droitesdansX}
Sous les hypoth\`eses du th\' eor\`eme \ref{thmF1=F2}, si une droite projective de $\mathbb P^3 (\mathbb C)$ appartient \`a l'hypersurface $\mathbb X$ d\'efinie par \eqref{defX} et contient un point rationnel, elle est d\'efinie par des \'equations
\begin{equation}\label{system1}
\begin{cases}
X_1&= u_1X_3+u_2X_4\\
X_2&=u_3 X_3 +u_4X_4,
\end{cases}
\end{equation}
o\`u l'un au moins des $u_i$ est irrationnel. 
\end{lem}

\begin{proof}[D\'emonstration]

Pour ${\bf a}= (a_1,a_2,a_3, a_4)$ et ${\bf b }= (b_1,b_2,b_3,b_4)$ deux quadruplets non nuls et non proportionnels de nombres complexes, on suppose que la droite projective $\mathbb D_{\bf a, \bf b}$ de $\mathbb P^3 (\mathbb C)$ d\' efinie par les \' equations
\begin{equation}\label{defDab}
\mathbb D_{\bf a, \bf b} \ :\ 
\begin{cases}
a_1X_1+a_2X_2+a_3X_3+a_4X_4 &=0\\
b_1X_1+b_2X_2+b_3X_3+b_4X_4 &=0
\end{cases}
\end{equation}
est contenue dans $\mathbb X$ et contient un point rationnel. 

\noindent $\bullet $ Si $a_1= a_2=0$ et $a_3 \not= 0$, on substitue $X_3 = -(a_4/a_3) X_4$ dans la deuxi\`eme \' equation de \eqref{defDab} qui devient ainsi

\begin{equation}\label{b1X1+b2X2=}
b_1 X_1+b_2 X_2 = \frac{a_4b_3 -a_3b_4}{a_3} X_4.
\end{equation}

$\diamond$ Si $a_4b_3-a_3 b_4 =0$, cela signifie que $\mathbb X$ contient une droite de la forme 
\begin{equation}
\label{(00a3a4)}
\begin{cases} a_3X_3 +a_4 X_4= 0\\
b_1X_1 +b_2 X_2 =0.
\end{cases}
\end{equation}
Exploitant la forme particuli\`ere de l'\' equation \eqref{defX} d\' efinissant $\mathbb X$, on d\'eduit 
$$
 b_1 X_1 +b_2 X_2 \text{ divise } F_1 (X_1, X_2) \text{ et } a_3X_3 +a_4 X_4 \text{ divise } F_2 (X_3, X_4).
$$
Compte tenu de l'hypoth\`ese sur les $F_i$, ces conditions de divisibilit\' e contredisent l'hypoth\`ese qu'il y a un point rationnel sur la droite d\'efinie par \eqref{(00a3a4)}.

$\diamond$ Si $a_4b_3-a_3b_4 \not= 0$, on remplace $X_3$ par la valeur donn\' ee par la premi\`ere \' equation de \eqref{defDab} et $X_4$ par la valeur donn\' ee en \eqref{b1X1+b2X2=} conduisant \`a
$$
F_1(X_1,X_2) = F_2(X_3,X_4) =(a_4b_3-a_3b_4)^{-d} F_2 (-a_4, a_3) (b_1X_1+b_2X_2)^d,
$$
ce qui contredit l'hypoth\`ese que le discriminant de $F_1$ est non nul.

\noindent $\bullet $ Si $a_1= a_2=0$ et $a_4 \not= 0$, par le m\^eme type de raisonnement suivi pr\' ec\' edemment, on est ramen\' e au cas o\`u $(a_1,a_2)\not= (0,0)$.

\noindent $\bullet$ Par sym\'etrie, on suit le m\^eme raisonnement dans les trois cas suivants : si $a_3=a_4=0$, si $b_1=b_2=0$ ou si $b_3=b_4 =0$.

\noindent $\bullet$ En conclusion de la discussion pr\' ec\' edente, nous avons prouv\' e que si $\mathbb D_{\bf a, \bf b}$, contenue dans $\mathbb X$, poss\`ede un point rationnel, on a n\' ecessairement 
\begin{equation}\label{aibinot=0}
(a_1, a_2), \ (a_3,a_4), \ (b_1, b_2), \text{ et } (b_3,\, b_4) \text{ sont } \not=(0,0).
\end{equation}
\noindent $\bullet$ Supposons maintenant $a_1b_2= a_2 b_1$ et $a_1\not= 0$. Multipliant la premi\`ere \' equation de \eqref{defDab} par $-b_1$ et la seconde par $a_1$, on obtient que dans ce cas le syst\`eme 
d'\'equations d\'efinissant $\mathbb D_{\bf a, \bf b}$ est \' equivalent \`a
$$
\begin{cases}
a_1X_1+a_2X_2+a_3X_3+a_4X_4 &=0\\
 (a_1b_3-a_3b_1)X_3+(a_1b_4-a_4 b_1)X_4 &=0.
\end{cases}
$$
Or \eqref{aibinot=0} a \'elimin\'e le cas $(b_1,b_2)=(0,0)$. On en d\'eduit que l'hypoth\`ese $a_1b_2= a_2 b_1$ et $a_1\not= 0$ est impossible. 

\noindent $\bullet$ Supposons maintenant $a_1b_2= a_2 b_1$ et $a_2\not= 0$. Mais ce cas est impossible par un raisonnement identique. Puisque par \eqref{aibinot=0} le cas $(a_1, a_2) =(0,0) $ est interdit,
on est ramen\' e \`a supposer que $a_1b_2\not= a_2b_1.$ 

\noindent $\bullet $ Pour finir on suppose donc que $a_1b_2\not= a_2b_1$. Par r\' esolution d'un syst\`eme (2,2) en les inconnues $X_1$ et $X_2$ et de d\' eterminant non nul,
on voit que le syst\`eme \eqref{defDab} est \'equivalent \`a
\begin{equation*} 
\begin{cases}
X_1&= u_1X_3+u_2X_4\\
X_2&=u_3 X_3 +u_4X_4,
\end{cases}
\end{equation*}
o\`u les $u_i$ sont des nombres complexes.
On a donc l'\' egalit\' e
\begin{equation}\label{F2=F1u1}
F_2 (X_3, X_4) = F_1 (u_1 X_3+ u_2 X_4, u_3X_3+u_4X_4).
\end{equation}

$\diamond$ Si $\det \begin{pmatrix} u_1& u_2\\ u_3 &u_4\end{pmatrix}=0$, cela signifie que par exemple, on a, pour un certain $\lambda$ complexe, l'\' egalit\' e $ u_1X_3+u_2X_4 = \lambda (u_3 X_3+u_4X_4)$ 
donc en reportant, on d\'eduit l'\' egalit\' e
$$
F_2 (X_3, X_4) = (u_3 X_3+u_4X_4)^d F_1 (\lambda, 1),
$$
ce qui contredit l'hypoth\`ese de non nullit\' e du discriminant de $F_2$.

$\diamond$ Si $\det \begin{pmatrix} u_1& u_2\\ u_3 &u_4\end{pmatrix}\not= 0$, par \eqref{F2=F1u1} on voit que les formes $F_1$ et $F_2$ sont isomorphes par un changement de variables
lin\'eaires \`a coefficients complexes. L'hypoth\`ese de non isomorphisme, sur ${\rm Gl}(2,\mathbb Q)$, de $F_1$ et $F_2$ implique que parmi les $u_i$
l'un au moins est irrationnel. 

Ceci termine la d\'emonstration du lemme \ref{lemme:droitesdansX}.
\end{proof}
\subsubsection{Preuve de la proposition \ref{N2<<}}
On sait que pour tout $d\geq 3$, il existe un entier $\ell (d)$, tel que, toute surface lisse de $ \mathbb P^3 (\mathbb C)$ de degr\' e $d$ contient au plus $\ell (d)$ droites. Pour des \'etudes fines concernant cette constante $\ell (d)$ on se reportera \`a \cite{Se} et \`a \cite{BoSa} par exemple.

 Gr\^ace au lemme \ref{lemme:droitesdansX}, on est ramen\' e \`a d\' enombrer l'ensemble des quadruplets d'entiers $(x_1,\dots, x_4)$ avec $\max \vert x_i\vert \leq B$, v\' erifiant 
 \begin{equation*}
\begin{cases}
x_1&= u_1x_3+u_2x_4\\
x_2&=u_3 x_3 +u_4x_4,
\end{cases}
\end{equation*}
sachant que l'un au moins des $u_i$ est irrationnel. Disons que c'est $u_1$. 

-- si $\dim_{\mathbb Q} (1, u_1, u_2) =3$, la seule solution en $(x_1, x_3, x_4)\in \mathbb Z^3$ de l'\' equation $x_1=u_1x_3+u_2x_4$ est $(0,0,0)$,

-- si $\dim_{\mathbb Q} (1, u_1, u_2) =2$, on exprime $u_2= a +b u_1$, avec $a$ et $b$ rationnels et on est ramen\'e \`a l'\' equation
$$
x_1-ax_4= (x_3+bx_4) u_1,
$$
qui implique $x_1-ax_4=x_3+bx_4=0$. Donc le syst\`eme \eqref{system1} admet $O (B)$ quadruplets $(x_1,\dots, x_4)$ solutions de hauteur inf\' erieure \`a $B$. 

Ceci termine la preuve de la proposition \ref{N2<<}.

\section{Quelques propri\' et\' es des formes cyclotomiques}
Dans ce paragraphe nous prouvons quelques r\' esultats g\' en\' eraux concernant les formes cyclotomiques $\Phi_n (X,Y)$ dont le degr\' e est $d=\varphi (n).$ 
Ces divers r\' esultats seront n\' ecessaires lors de la preuve des th\' eor\`emes \ref{thmcentral}, \ref{thmcentral2}, \ref{Thm:ovalecyclotomique} et \ref{thmcentral3}.  
En particulier les r\'esultats des sections \S \ref{confinemodp}, \S \ref{confinemod9} et \S \ref{confinemod4}. ne seront utilis\' es que pour la preuve du Th\' eor\`eme \ref{thmcentral3}. 
Nous rappelons d'abord
plusieurs formules classiques sur les $\Phi_n (X,Y)$. Ces formules ne sont que la version homog\`ene des formules correspondantes sur les $\phi_n (x)$.

Nous rappelons certaines notations et conventions : si $n\geq 1$ est un entier, on d\' esigne par $\mu (n)$ la valeur de la fonction de M\" obius, $\omega (n)$ est le nombre de
facteurs premiers distincts de $n$, $\kappa (n)$ est le {\it radical } de $n$, c'est--\`a--dire le produit de 
tous les premiers divisant $n$, $d(n)$ le nombre de diviseurs. On dit que deux nombres rationnels $u$ et $v$ sont congrus modulo le nombre premier $p$
si on a $u-v\in p\mathbb Z_p$, o\`u $\mathbb Z_p$ est l'anneau des entiers $p$--adiques.

Si $n\geq 2$ est factoris\' e en $n:=p^r m$, avec $p\nmid m$ et $r \geq 1$, 
la forme $\Phi_n$ v\' erifie les identit\' es suivantes 
\begin{equation*} 
\Phi_n(X,Y) =\prod_{d \mid n} (X^d-Y^d)^{\mu (n/d)},
\end{equation*}
\begin{equation}\label{list2}
\Phi_n(X,Y) =\frac{\Phi_m(X^{p^r}, Y^{p^r})}{\Phi_m (X^{p^{r-1}},Y^{p^{r-1}})},
\end{equation}
et
\begin{equation*} 
\Phi_n (X,Y) = \Phi_{pm}(X^{p^{r-1}}, Y^{p^{r-1}}).
\end{equation*}
Par it\' eration de cette derni\`ere formule, on parvient \`a
\begin{equation}\label{list4}
\Phi_n (X,Y) = \Phi_{\kappa (n)} (X^{n/\kappa (n)}, Y^{n/\kappa (n)}).
\end{equation}
Nous rappelons quelques valeurs de $\phi_n$ en certains points :
\begin{equation}\label{phin(1)=}
\phi_n (1)=
\begin{cases}
0&\text{ si } n=1,\\
p& \text{ si } n=p^k\ (k\geq 1),\\
1& \text { si } \omega (n) \geq 2,
\end{cases}
\end{equation}
et 
\begin{equation}\label{phin(-1)=}
\phi_n (-1)=
\begin{cases}
-2&\text{ si } n=1,\\
\phi_{n/2} (1) &\text{ si } n\geq 2, \, 2 \Vert n,\\
1& \text { si } n\geq 3, \, 2\nmid n,\\
1&\text{ si } n\geq 4, \, 4\vert n, \, n\not= 2^\ell,\\
2&\text{ si } n\geq 4, \, n=2^\ell.
\end{cases}
\end{equation}
Pour majorer les coefficients de $\Phi_n$ nous utiliserons le r\' esultat suivant, d\^u \`a P. Bateman \cite[p.1181]{Ba}.
Quand $P$ est un polyn\^ome, nous d\'esignons par $\rmL(P)$ ({\it longueur de $P$}) la somme des valeurs absolues des coefficients de $P$.

\begin{lem} 
Pour tout $n\geq 1$, on a 
$$
\rmL(\phi_n)\le 
n^{d(n)/2}.
$$ 
\end{lem}

 Les majorations classiques des fonctions arithm\'etiques $d(n)$ et $\varphi(n)$ impliquent alors que pour tout $\varepsilon>0$ et pour $n$ suffisamment grand, on a 
 \begin{equation}\label{equation:Bateman}
 \varphi(n)\rmL(\phi_n)\le e^{n^\varepsilon}.
\end{equation}

\subsection{Propri\' et\' es de confinement modulo $p$}\label{confinemodp} Nous prouvons que pour tout $a$ et $b$ entiers $\Phi_n (a,b)$ est, pour tout $m$ divisant $n$, restreint \`a quelques classes de congruence modulo $m$.
\begin{prop}\label{prop:confinmodp} Soient $n\geq 2$ et $p$ un premier divisant $n$. Alors pour tout $a$ et $b$ de $\mathbb Z$ on a 
\begin{equation}\label{???}
\Phi_n (a,b) \equiv 0,\ 1\bmod p.
\end{equation}
\end{prop}
\begin{proof}[D\' emonstration] Remarquons d'abord que lorsque $p=2$ ou lorsque $b\equiv 0\bmod p$, l'\' enonc\' e pr\' ec\'edent est trivial. Enfin on peut se restreindre au cas 
$$
n \text { sans facteur carr\'e.}
$$
 C'est une cons\' equence de l'inclusion des images 
$\Phi_n (\mathbb Z, \mathbb Z) \subset \Phi_{\kappa (n)} (\mathbb Z, \mathbb Z)$, qui se d\' eduit directement de \eqref{list4}. Commen\c cons par le cas o\`u $n$ est un nombre premier.
On a 
\begin{lem}\label{defini1}
Soient $p\geq 3$ un nombre premier, $a$ et $b$ deux entiers. Alors on a les congruences 
\begin{enumerate}
\item Si $a\not\equiv b \bmod p$, on a
$$
\Phi_p (a,b) \equiv 1 \bmod p,
$$
\item Si $a\equiv b \bmod p, $ on a 
$$
\Phi_p(a,b) \equiv p a^{p-1} \bmod p^2.
$$

 \end{enumerate}

\end{lem}
\begin{proof}[D\' emonstration du lemme \ref{defini1}] C'est une cons\'equence de la formule $a^p\equiv a \bmod p$. Dans le premier cas, 
on \' ecrit $$\Phi_p (a,b) = (a^p -b^p)/ (a-b) \equiv (a-b)/(a-b) \equiv 1 \bmod p.$$
Dans le second cas, on \' ecrit
$$
\Phi_p (a, b) = a^{p-1} + a^{p-2} b + \cdots + ab^{p-2}+ b^{p-1}, 
$$
on pose $b=a+pt$ avec $t$ entier et on d\' eveloppe suivant la formule du bin\^ome pour obtenir le r\' esultat.
\end{proof}
Poursuivons la d\'emonstration de la proposition \ref{prop:confinmodp}.
On suppose donc que $n$ est sans facteur carr\' e et on pose $$n=pm=pp_2\cdots p_t.$$ Par it\' eration de \eqref{list2}, on a l'\' egalit\' e
\begin{equation}\label{Phin=}
\Phi_n(a,b) = \prod_{m_1 \mid m} \Phi_p (a^{m_1}, b^{m_1})^{\mu (m/m_1)}.
\end{equation}
Ceci nous am\`ene \`a d\' ecomposer le produit \`a droite de l'\'egalit\' e \eqref{Phin=} en 
$$
\Phi_n(a,b) =
\Phi_n^\dag (a,b) \tilde \Phi_n(a,b),
$$
o\`u $\Phi^\dag_n$ correspond \`a la condition $a^{m_1}\not\equiv b^{m_1} \bmod p$ et $\tilde \Phi$ le produit compl\'ementaire.
Par le lemme \ref{defini1}, on a $\Phi_p(a^{m_1}, b^{m_1}) \equiv 1 \bmod p$ si et seulement si $a^{m_1} \not \equiv b^{m_1} \bmod~p.$
Ceci implique que $\Phi_n^\dag (a,b) $ est un nombre rationnel qui est produit et quotient d'entiers congrus \`a $1 \bmod p$. C'est donc un nombre rationnel congru \`a $1 \bmod p$. 

Par d\' efinition, on a l'\' egalit\' e 
\begin{equation}\label{deftildePhi}
 \tilde \Phi_n (a,b) := \prod_{\substack{m_1 \mid m\\ (a/b)^{m_1}\equiv 1 \bmod p} }\Phi_p (a^{m_1}, b^{m_1})^{\mu (m/m_1)}. 
\end{equation}
D'apr\`es ce qui pr\' ec\`ede, pour compl\' eter la preuve de la congruence \eqref{???}, il reste \`a prouver que $\tilde \Phi_n (a,b)$ est un rationnel congru \`a $0$ ou $1\bmod p$.

 Soit $\ell$ l'ordre de $(a/b)$ modulo $p$. Ainsi $\ell$ divise $p-1$ et le produit apparaissant dans \eqref{deftildePhi} est sur les $m_1$ tels que
 $$
 \ell \text{ divise } m_1 \text{ et } m_1 \text{ divise } m.
 $$
 Ce produit est vide lorsque $\ell \nmid m$; c'est par exemple le cas si $a\not\equiv b \bmod p$ et si $(p-1, m) =1$. On suppose donc que $\ell \mid m$. 
 Quitte \`a r\' eindicer les $p_j$, on peut supposer que $\ell$ est de la forme
 $$
 \ell =p_2\cdots p_r,
 $$
 avec $1\leq r\leq t$, avec la convention que $\ell =1$ si $r=1$. Posant alors $h= p_{r+1}\cdots p_t = m/\ell$ et $m_1=\ell m_2$, on r\' ecrit la d\' efinition \eqref{deftildePhi} comme
 \begin{equation}\label{deftildePhibis}
 \tilde \Phi_n (a,b) := \prod_{ m_2\mid h }\Phi_p (a^{\ell m_2}, b^{\ell m_2})^{\mu (h/m_2)}. 
\end{equation}

 On articule la discussion suivant plusieurs cas.

 \hskip 1cm $\bullet$ si $h=1$, le produit apparaissant dans \eqref{deftildePhibis} ne contient qu'un seul terme \`a savoir $\Phi_p (a^\ell, b^\ell)$. Il est congru \`a $0 \bmod p,$ d'apr\`es le lemme \ref{defini1}.2.
 
 \vskip .3cm
 
 \hskip 1cm $\bullet$ si $h\not= 1$, en utilisant la formule de M\" obius, on \' ecrit \eqref{deftildePhibis} sous la forme
 $$
 \tilde \Phi_n (a,b) = \prod_{ m_2\mid h }\Bigl(\frac{\Phi_p (a^{\ell m_2}, b^{\ell m_2})}{p}\Bigr)^{\mu (h/m_2)}, 
 $$
 o\`u maintenant chaque fraction du produit est congrue \`a $a^{\ell m_2 (p-1)}$ modulo $p$ (lemme \ref{defini1}.2). Ainsi chacune de ces fractions est 
 un entier premier \`a $p$. Le nombre rationnel $ \tilde \Phi_n (a,b)$ v\'erifie donc
 $$
 \tilde \Phi_n (a,b)\equiv a^{\ell m_2 (p-1)\sum_{m_2\mid h} \mu (h/m_2)}\equiv 1 \bmod p,
 $$
 en utilisant de nouveau la formule de M\" obius. Ceci termine la preuve de la proposition \ref{prop:confinmodp}.
\end{proof}

\subsection{Propri\'et\' es de confinement modulo $9$}\label{confinemod9}
Par la proposition \ref{prop:confinmodp}, on sait que si $3\mid n$, on a $\Phi_n (a,b)\equiv 0,\, 1 \bmod 3$. Ce n'est pas assez satisfaisant pour la future application. Nous utiliserons une version plus pr\' ecise avec 
 la
\begin{prop}\label{prop:confinmod9} Soit $k\geq 2$. Alors pour tout $a$ et tout $b$ entiers, on a la congruence
$$
\Phi_{3^k}(a,b) \equiv 0,\, 1,\ 3\bmod 9.
$$
\end{prop}
\begin{proof}[D\' emonstration] Les cubes modulo $9$ forment l'ensemble
$$\mathfrak Q (9) :=\{0,\ 1,\, -1\}.
$$
Or $$\Phi_3 (u,v) =u^2+uv+v^2.$$ Si $u$ et $v$ parcourent l'ensemble $\mathfrak Q (9)$ on voit que $\Phi_3 (u,v)$ parcourt
l'ensemble
$$
\mathcal E = \{0, 1, 3\bmod 9\}.
$$
Pour terminer la preuve de la proposition, il suffit d'utiliser le fait 
que
$$
\Phi_{3^k} (a,b) = \Phi_3 (a^{3^{k-1}}, b^{3^{k-1}}),
$$
cons\' equence de \eqref{list4}.
\end{proof}
\subsection{Propri\' et\' es de confinement modulo $4$.}\label{confinemod4}

Nous envisageons maintenant le cas o\`u $n$ est pair. Par la remarque \eqref{nnequiv}, on peut m\^eme supposer que $4 \mid n$ et on \' ecrit que $2^k \Vert n$ avec $k\geq 2$.

\vskip .2cm
\begin{enumerate}
\item Soit $a$ pair et $b$ impair. Par application it\' er\' ee de \eqref{list2}, on voit $\Phi_n (X,Y)$ est produit et quotient de polyn\^omes de la forme $\Phi_{2^k} (X^\alpha, Y^\alpha)$ o\`u $\alpha$ et $\beta$ sont des nombres impairs. Puisque $\Phi_{2^k}(X,Y) = X^{2^{k-1}} +Y^{2^{k-1}}$ et $k\geq 2$, on d\' eduit que $\Phi_{2^k}
(a^\alpha, b^\beta)\equiv 0+1 \equiv 1 \bmod 4$ et, par cons\'equent, que $$\Phi_n (a,b)\equiv 1 \bmod 4.$$
 
\item Soit $a$ et $b$ pairs. Puisque $\Phi_n (X,Y)$ est somme de mon\^omes de la forme $c_{\mu, \nu} X^\mu Y^\nu$ avec $\mu+\nu = \varphi (n)$ et $c_{\mu, \nu}$ entier, on voit que $4\mid c_{\mu, \nu} a^\mu b^\nu$, d'o\`u
$$\Phi_n (a,b) \equiv 0 \bmod 4.$$
\item Supposons $a\equiv b\equiv 1 \bmod 4$. On \' ecrit $\Phi_n(a,b) = b^{\varphi (n)} \phi_n (a/b) \equiv \phi_n (1) \bmod 4$. Et, d'apr\`es \eqref{phin(1)=}, ceci vaut $2\bmod 4$ si $n=2^k$ avec $k\geq 2$
ou $1\bmod 4$ si $n \not= 2^k.$ D'o\`u 
$$
\Phi_n(a,b) \equiv 1,\, 2 \bmod 4.
$$
\item Supposons $a\equiv -b \equiv 1 \bmod 4$. On \' ecrit $\Phi_n(a,b) = b^{\varphi (n)} \phi_n (a/b) \equiv \phi_n (-1) \bmod 4$. Il suffit d'appliquer
les deux derni\`eres lignes de \eqref{phin(-1)=} pour 
conclure que l'on a, dans ce cas 
$$
\Phi_n(a,b) \equiv 1,\, 2 \bmod 4.
$$
\item Supposons $a\equiv - 1 \bmod 4$. On utilise la relation $\Phi_n(a,b)=\Phi_n(-a,-b)$ valable pour $n\geq 3$. 
\end{enumerate}

Nous rassemblons ces divers r\' esultats sous la forme de la
\begin{prop}\label{prop:confinmod4} Soit $n$ un entier divisible par $4$. Alors pour tout $a$ et tout $b$ entiers, on a la congruence
$$
\Phi_{n}(a,b) \equiv 0,\, 1,\ 2\bmod 4.
$$
\end{prop}

\subsection{ Automorphismes des formes cyclotomiques} \label{autoforcyclo}
Soit $\Phi_n(X,Y)$ une forme cyclotomique de degr\'e $d=\varphi (n)$. Par la d\' efinition \eqref{defauto}, rechercher les automorphismes de $\Phi_n$
consiste \`a rechercher les matrices $U$ de ${\rm Gl}(2, \mathbb Q)$
$$
U =\begin{pmatrix} u_1 & u_2\\
u_3 & u_4
\end{pmatrix},
$$
telles que 
\begin{equation}\label{defauto*}
\Phi_n (X,Y) =\Phi_n (u_1X +u_2Y, u_3X+u_4 Y).
\end{equation}
Cette \'egalit\' e formelle entra\^ine que l'ensemble $\mathbb U_n$ des racines primitives $n$--i\`emes de l'unit\' e est stable par l'application $\mathcal H$ de 
$\widehat {\mathbb C} $ dans $\widehat {\mathbb C}$ d\' efinie par 
\begin{equation}\label{defH*} z\mapsto 
\mathcal H(z) :=\frac{u_1z +u_2}{u_3z+u_4}.
\end{equation}
Si $\mathbb U_n$ a au moins trois \' el\'ements (c'est--\`a--dire $n\geq 5$ et $n\not=6$), il y a un cercle et un seul contenant $\mathbb U_n$. Il s'agit du cercle $\mathbb S^1$ et celui--ci est stable par $\mathcal H$. Les transformations de $\widehat {\mathbb C}$ de la forme $(az+b)/(cz+d)$ (avec $(a, b, c, d) \in \mathbb C^4$ et $ad-bc\not=0$) laissant $\mathbb S^1$ globalement invariant sont connues : il s'agit des transformations $z\mapsto \rho z$ et $z\mapsto \rho/z$ avec $\rho $ nombre complexe de module $1$. Ainsi la fonction $\mathcal H$ d\' efinie en \eqref{defH*} a n\' ecessairement une des quatre formes
$$
\mathcal H (z) =z, \ -z,\ 1/z, \ -1/z,
$$
puisque les $u_i$ sont des rationnels. Enfin pour tout $n \geq 1$, on a l'\' equivalence $$\xi \in \mathbb U_n\iff 1/\xi \in \mathbb U_n$$ et seulement pour $n\equiv 0 \bmod 4$, l'\' equivalence $$\xi \in \mathbb U_n \iff -\xi \in \mathbb U_n.$$
Nous voyons donc que si
\begin{enumerate}
\item si $n\equiv 0\bmod 4$, et $n\geq 8$, on a $ \frac{u_1z+u_2}{u_3z+u_4} =z,\ -z,\ \frac{1}{z}, \ -\frac{1}{z},$
\item si $n \not\equiv 0 \bmod 4$, $n\geq 5$ et $n\not= 6$, on a $\frac{u_1z+u_2}{u_3z +u_4} = z,\ -z.$
\end{enumerate}
Revenant \`a la d\' efinition \eqref{defauto*} et rappelant que $\Phi_n$ est une forme homog\`ene de degr\'e $\varphi (n)$ on obtient les matrices $U$: 
\begin{enumerate}
\item Pour $n\equiv 0\bmod 4$, et $n\geq 8$, on a 
$$
U = \begin{pmatrix}\pm1 & 0\\ 0 & \pm 1 \end{pmatrix}, \ \begin{pmatrix}\pm1 & 0\\ 0 & \mp 1 \end{pmatrix}, \ \begin{pmatrix}0 & \pm 1\\ \pm 1 & 0 \end{pmatrix},\ \begin{pmatrix}0 & \pm 1\\ \mp 1 & 0 \end{pmatrix}$$
\item Pour $n \not\equiv 0 \bmod 4$, $n\geq 5$ et $n\not= 6$, on a 
$$U= \begin{pmatrix}\pm1 & 0\\ 0 & \pm 1 \end{pmatrix}, \ \begin{pmatrix}0 & \pm 1\\ \pm 1 & 0 \end{pmatrix}.$$
\end{enumerate}
Les petites valeurs de $n$ (celles v\' erifiant $\varphi (n) =2$) se font directement gr\^ace aux
formes explicites donn\' ees en \eqref{Phi3Phi4Phi6}. En notant $\mathbb D_k$ le groupe di\' edral \`a $2k$ \'el\' ements, on obtient la 
 \begin{prop} 
 Soit $n \geq 3$ un entier. Alors le groupe des automorphismes $ {\rm Aut} \Phi_n$ de $\Phi_n (X,Y)$
est
$$
 {\rm Aut}\, \Phi_n =
 \begin{cases} \mathbb D_4 & \text{ si $4$ divise $n$}, \\
 \mathbb D_2& \text{ dans le cas contraire}.
 \end{cases} 
 $$
\end{prop}

Nous en d\'eduisons:
\begin{cor}\label{corollaire:wPhin}
Pour $n \geq 3$, on a $W_{\Phi_n}=w_n$ o\`u $w_n$ est d\'efini par \eqref{Equation:wn}.
\end{cor}

\begin{proof}[D\'emonstration] 
Les groupes d'automorphismes des formes cyclotomiques 
sont constitu\'es de matrices \`a coefficients entiers. 
Ainsi, dans les d\' efinitions \eqref{defWF1} et \eqref{defWF2}, quand $F$ est une forme cyclotomique, les r\'eseaux $\Lambda$ et $\Lambda_i$ sont \'egaux \`a $\mathbb Z^2$,
 leur d\'eterminant vaut $1$.
\end{proof}

\subsection{Isomorphismes entre formes cyclotomiques} 
Rappelons que la d\'efinition d'isomorphisme entre deux formes binaires a \' et\' e donn\' ee en \eqref{isom}.
La proposition suivante caract\' erise les formes binaires cyclotomiques isomorphes.

\begin{prop}\label{prop:formescyclotomiquesisomorphes}
Soient $n_1$ et $n_2$ deux entiers positifs avec $n_1<n_2$. Les conditions suivantes sont \' equivalentes. 
\\
$(1)$ On a $\varphi(n_1)=\varphi(n_2)$ et les deux formes binaires cyclotomiques $\Phi_{n_1}$ et $\Phi_{n_2}$ sont isomorphes.
\\
$(2)$ Les deux formes binaires cyclotomiques $\Phi_{n_1}$ et $\Phi_{n_2}$ repr\'esentent les m\^emes entiers.
\\
$(3)$ $n_1$ est impair et $n_2=2n_1$.
\end{prop} 

Les formes binaires cyclotomiques $\Phi_n(X,Y)$ avec $\varphi(n)=d$ et $n$ non congru \`a $2$ modulo $4$ forment donc un syst\`eme complet de repr\' esentants des classes d'isomorphisme des formes binaires cyclotomiques de degr\'e $d$.

La d\' emonstration de la proposition \ref{prop:formescyclotomiquesisomorphes} utilisera le lemme suivant:

\begin{lem}\label{lem:formescyclotomiquesisomorphes}
Soit $n$ un entier positif. Le groupe de torsion du corps cyclotomique $\Q(\zeta_n)$ est cyclique, d'ordre $n$ si $n$ est pair, d'ordre $2n$ si $n$ est impair.
\end{lem} 

\begin{proof}[D\' emonstration du lemme \ref{lem:formescyclotomiquesisomorphes}]
Le groupe de torsion du corps cyclotomique $\Q(\zeta_n)$ est cyclique, d'ordre multiple de $n$. S'il est d'ordre sup\'erieur \`a $n$, alors il contient une racine primitive de l'unit\'e d'ordre $pn$, avec $p$ premier, dont le degr\' e est $\varphi(pn)$. On en d\'eduit $\varphi(pn)=\varphi(n)$, d'o\`u il r\'esulte que $p=2$ et que $n$ est impair. 
\end{proof}

\begin{proof}[D\' emonstration de la proposition \ref{prop:formescyclotomiquesisomorphes}]
\null\hfill\break
$(1)\Longrightarrow(3)$
Supposons $\Phi_{n_1}$ et $\Phi_{n_2}$ isomorphes avec $n_1<n_2$. Il existe une matrice
$$
\begin{pmatrix}
u_1&u_2\\u_3&u_4
\end{pmatrix}\in\GL_2(\Q)
$$
telle que 
$$
\zeta_{n_1}=\frac{u_1\zeta_{n_2}+u_2}{u_3\zeta_{n_2}+u_4}\cdotp
$$
On en d\' eduit que les corps cyclotomiques $\Q(\zeta_{n_1})$ et $\Q(\zeta_{n_2})$ co•\"incident, et le lemme \ref{lem:formescyclotomiquesisomorphes} donne le r\' esultat.

\noindent
$(3)\Longrightarrow(2)$. Si $n_1$ est impair et $n_2=2n_1$, alors $\phi_{n_2}(t)=\phi_{n_1}(-t)$, donc les deux formes binaires cyclotomiques $\Phi_{n_1}$ et $\Phi_{n_2}$ repr\' esentent les m\^emes entiers.

\noindent
$(2)\Longrightarrow(1)$. En utilisant les notations du th\' eor\`eme 
\ref{thm21} et du corollaire \ref{F1=F2}, nous avons, par hypoth\`ese les \' egalit\' es
\begin{equation}\label{R=R=R}
R_{\Phi_1} (N) = R_{\Phi_2} (N) =R_{\Phi_1, \Phi_2} (N),
\end{equation}
pour tout $N\geq 1$. Par le th\'eor\`eme \ref{thm21} la premi\`ere \' egalit\'e de \eqref{R=R=R} implique $\varphi (n_1) =\varphi (n_2).$ Enfin si $\Phi_{n_1}$ et $\Phi_{n_2} $
n'\' etaient pas isomorphes, le corollaire \ref{F1=F2} entra\^inerait que la deuxi\`eme \' egalit\' e de \eqref{R=R=R} serait impossible pour $N$ suffisamment grand.
\end{proof}

\subsection{R\' esultats auxiliaires de comptage}

Notre d\' emonstration du th\' eor\`eme \ref{thmcentral} au \S \ref{S:DemThmPpal} utilisera l'\' enonc\' e suivant \cite[Theorem 1.1]{FLW}.

\begin{thm} \label{Theorem:FLW}
Soit $m$ un entier positif et soient $n,x,y$ des entiers rationnels v\' erifiant $n\ge 3$, $\max\{|x|,|y|\}\ge 2$ et $\Phi_n(x,y)=m$. Alors 
$$
\max\{|x|,|y|\} \le \frac{2}{\sqrt{3}} \, m^{\frac{1}{\varphi(n)}}
\quad\hbox{et par cons\' equent,}\quad
\varphi(n) \le \frac{2}{\log 3} \log m.
$$
\end{thm}

Nous en d\' eduisons le

\begin{cor}\label{Corollaire:FLW}
Pour $d\ge 2$ et $N\ge 1$, on a la majoration
$$
\calA_d(N) \le 
29N^{\frac{2}{d}}(\log N)^{1.161}.
$$
\end{cor}

\begin{proof}[D\' emonstration du corollaire \ref{Corollaire:FLW}]
D'apr\`es le th\' eor\`eme \ref{Theorem:FLW}, les conditions $\varphi(n)\ge d$ et $\Phi_n(x,y)\le N$ impliquent $\max\{|x|,|y|\} \le \frac{2}{\sqrt{3}}N^{\frac{1}{d}}$. Notons que $\calA_d(N) =0$ pour $N=1$ et $N=2$. 
La condition $\max\{|x|,|y|\}\ge 2$ permet d'obtenir 
$\varphi(n) \le \frac{2}{\log 3} \log N$, donc $n\le 5.383 (\log N)^{1.161}$ (formule (1.1) de \cite{FLW}). Il en r\'esulte que le nombre de triplets $(n,x,y)$ tels que $\varphi(n)\ge d$ et $\Phi_n(x,y)\le N$ est major\'e par 
$$
\frac{16}{3} 5.383 N^{\frac 2 d}(\log N)^{1.161}.
$$
\end{proof}

\section{D\' emonstration des th\' eor\`emes \ref{thmcentral} et \ref{thmcentral2} }\label{S:DemThmPpal}

Pour $n$ entier avec $\varphi (n) \geq 4$ et $N \geq 1,$
on d\' esigne par $\mathcal B_n (N)$ l'ensemble
\begin{equation}\label{defcalB}
\mathcal B_n (N):=\bigl\{ m\leq N \mid m = \Phi_n (a,b)\text{ avec } \max (\vert a \vert , \vert b \vert )\geq 2\bigr\}.
\end{equation}
Par le th\' eor\`eme \ref{Theorem:FLW}, on a l'implication
$$
\mathcal B_n (N) \not= \emptyset \Rightarrow \varphi (n) \ll \log N,
$$
soit encore
$$
\mathcal B_n (N) \not= \emptyset \Rightarrow n \ll \log N \log \log \log N,
$$
uniform\' ement pour $N>10$. Ainsi, par la d\' efinition de $\mathcal A_d (N)$ et par la restriction \eqref{nnequiv}, nous avons l'\' egalit\' e
\begin{equation}\label{reunion}
\mathcal A_d (N)= \Bigl\vert \bigcup_{\substack{n\not\equiv 2 \bmod 4 \\ \varphi (n) \geq d} } \mathcal B_n (N) \Bigr\vert,
\end{equation}
o\`u cette r\' eunion porte sur un nombre fini de $n$.

Le terme principal dans l'estimation du cardinal de $\mathcal A_d(N)$ sera
$$
 \sum_{\substack{ n\not\equiv 2 \bmod 4\\ \varphi (n) =d} }\bigl\vert \mathcal B _n (N)\bigr\vert .
 $$
L'\' egalit\' e $\bigl\vert \mathcal B_n (N)\bigr\vert = R_{\Phi_n} (N)$ permet d'appliquer le th\' eor\`eme \ref{thm21} \`a chacun des termes de cette somme:
$$
\bigl\vert \mathcal B _n (N)\bigr\vert=
 A_{\Phi_n} \, W_{\Phi_n} N^\frac 2d + O_{\Phi_n,\varepsilon} (N^{\beta_d ^*+\varepsilon}). 
 $$
Gr\^ace au corollaire \ref{corollaire:wPhin}, on a $W_{\Phi_n}=w_n$. On obtient
\begin{equation}\label{Eq:TermePrincipal}
 \sum_{\substack{ n\not\equiv 2 \bmod 4\\ \varphi (n) =d} }\bigl\vert \mathcal B _n (N)\bigr\vert 
 =
 C_d N^\frac 2d +O (N^{\beta_d^* + \varepsilon})
\end{equation}
avec la valeur de $C_d$ annonc\'ee dans la formule \eqref{defCd}. 

\subsection{Minoration de $\mathcal A_d (N)$}
En restreignant le nombre de termes dans l'\' egalit\'e \eqref{reunion}, on a la minoration
\begin{align} 
\mathcal A_d (N)& \geq \Bigl\vert \bigcup_{\substack{n\not\equiv 2 \bmod 4 \\ \varphi (n) =d} } \mathcal B_n (N) \Bigr\vert \nonumber\\
&\geq \sum_{\substack{ n\not\equiv 2 \bmod 4\\ \varphi (n) =d} }\bigl\vert \mathcal B _n (N)\bigr\vert -\underset {\substack{n_1< n_2\\\varphi (n_1) =\varphi (n_2) =d\\ n_1,\ n_2 \not\equiv 2 \bmod 4} } {\sum \sum} \Bigl\vert \mathcal B_{n_1} (N) \cap \mathcal B_{n_2} (N) \Bigr\vert. \label{Ad>}
\end{align}
La premi\`ere partie du membre de droite de \eqref{Ad>} est trait\'ee dans \eqref{Eq:TermePrincipal}.

 Pour la seconde partie, on \' ecrit
l'\' egalit\' e $\bigl\vert \mathcal B_{n_1} (N) \cap \mathcal B_{n_2} (N) \bigr\vert =R_{\Phi_{n_1}, \Phi_{n_2}} (N)$. Par la proposition \ref{prop:formescyclotomiquesisomorphes} les formes $\Phi_{n_1}$ et $\Phi_{n_2}$ ne sont pas isomorphes.
Le corollaire \ref{F1=F2} donne ainsi la majoration 
$$
\bigl\vert \mathcal B_{n_1} (N) \cap \mathcal B_{n_2} (N) \bigr\vert = O (N^{\eta_d +\varepsilon}).
$$
En conclusion, nous avons prouv\' e la minoration suivante de $\mathcal A_d (N) $: 
\begin{equation*} 
\mathcal A_d (N) \geq C_d N^\frac 2d - O (N^{\beta_d^* + \varepsilon}) - O (N^{\eta_d+\varepsilon}),
\end{equation*}
qui se simplifie en
\begin{equation}\label{minorationAd}
\mathcal A_d (N) \geq C_d N^\frac 2d - O (N^{\eta_d+\varepsilon}),
\end{equation}
puisque, d'apr\`es \eqref{defeta} et \eqref{defbeta*}, on a pour tout $d\geq 4$ pair, l'in\' egalit\' e 
\begin{equation}\label{etad>beta*}
\eta_d \geq \beta^*_d.
\end{equation} 
\subsection{Majoration de $\mathcal A_d (N)$}
On \' ecrit maintenant
$$
\mathcal A_d (N)\leq \Bigl\vert \bigcup_{\substack{n\not\equiv 2 \bmod 4 \\ \varphi (n)= d} } \mathcal B_n (N) \Bigr\vert +\mathcal A_{d^\dag} (N),
$$
dont on d\'eduit la majoration 
\begin{equation}\label{Ad<}
\mathcal A_d (N)\leq \sum_{\substack{n\not\equiv 2 \bmod 4 \\ \varphi (n)= d} } \bigl\vert \mathcal B_n (N) \bigr\vert 
+\mathcal A_{d^\dag} (N).
\end{equation}
Le premier terme de cette majoration a d\' eja \' et\' e trait\' e en \eqref{Eq:TermePrincipal}. Pour majorer $\mathcal A_{d^\dag} (N)$ on utilise le corollaire \ref{Corollaire:FLW} avec $d$ remplac\'e par $d^\dag$. Revenant en \eqref{Ad<}, on a donc prouv\'e la majoration 
\begin{equation}\label{premierpas}
\mathcal A_d (N)\leq C_d N^\frac 2d + O (N^{\beta_d^* +\varepsilon}) + O\bigl(N^\frac{2}{d^\dag}( \log N)^{1.\, 161}\bigr).
\end{equation}
Cette formule appliqu\' ee en rempla\c cant $d$ par $d^\dag$ donne la majoration 
$$
\mathcal A_{d^\dag} (N)\ll N^\frac{2}{d^\dag},
$$
o\`u il n'y a plus de puissance de $\log N$ parasite. Reportant cette derni\`ere majoration dans \eqref{Ad<}, l'in\' egalit\' e \eqref{premierpas} est am\' elior\'ee en 
\begin{equation}\label{deuxiemepas}
\mathcal A_d (N)\leq C_d N^\frac 2d + O (N^{\beta_d^* +\varepsilon}) + O(N^\frac{2}{d^\dag} ).
\end{equation}
En combinant \eqref{minorationAd}, \eqref{deuxiemepas} et l'in\'egalit\'e \eqref{etad>beta*},  on termine la preuve de \eqref{central}. Les preuves des th\' eor\`emes \ref{thmcentral} et \ref{thmcentral2} sont compl\`etes.

\section{Preuve du th\' eor\`eme \ \ref{Thm:ovalecyclotomique}}

\subsection{R\'egion fondamentale d'une forme binaire d\'efinie positive}\label{SS:regionfondamentale}

Quand $F$ est une forme binaire d\'efinie positive de degr\'e $d$, on d\'esigne par $\calO(F)$ l'ensemble des $(x,y)\in\R^2$ tels que $F(x,y)\le 1$. 
C'est un compact du plan euclidien rapport\'e au rep\`ere orthonorm\'e $(O, \vec i, \vec j)$. Il est d\'elimit\'e par une courbe alg\'ebrique de degr\'e $d$ et de classe $C^\infty$, d'\'equation $F(x,y)=1$. 
Rappelons (\S~\ref{section:introd}) que $A_F$ d\'esigne l'aire de $\calO(F)$. Le changement de variable $(x,y)\to (t,y)$ avec $x=ty$ donne
 $$
A_F=
 \iint_{F(x,y)\le 1} \rmd x \rmd y
=
 \int_{-\infty}^{+\infty} \frac { \rmd t}{F(t,1)^{2/d}}\cdot
 $$
 Quand $F$ a ses coefficients alg\'ebriques, ce nombre est une p\'eriode au sens de Kontsevich -- Zagier. L'article \cite{Bean} est consacr\'e au calcul de $A_F$. 
 
On d\'esigne par $\rmL(F)$ la longueur du polyn\^ome $F(X,1)$ et par $\rmm(F)$ le minimum de la fonction $F(t,1)$ sur $\R$. Comme $F$ est homog\`ene de degr\'e $d$, on a, pour tout $(x,y)\in\R^2$, 
$$
\rmm(F)\max\{|x|,|y|\}^d \le F(x,y)\le \rmL(F)\max\{|x|,|y|\}^d.
$$ 
Il en r\'esulte que $\calO(F)$ contient le carr\'e centr\'e en $O$ de c\^ot\'e $\rmL(F)^{-1/d}$, \`a savoir
$$
\{(x,y)\in\R^2\; \mid\; \max\{|x|,|y|\}\le \rmL(F)^{-1/d}\},
$$ 
et qu'il est contenu dans le carr\'e centr\'e en $O$ de c\^ot\'e $\rmm(F)^{-1/d}$:
$$
\{(x,y)\in\R^2\; \mid\; \max\{|x|,|y|\}\le \rmm(F)^{-1/d}\}.
$$ 
Par cons\'equent,
 $$
 4 \rmL(F)^{-2/d} \le A_F\le 4 \rmm(F)^{-2/d}.
 $$

\subsection{Le domaine fondamental cyclotomique $\calO_n$ pour $n\ge 3$}
 Pour $n\ge 3$, 
 $\calO_n=\calO(\Phi_n)$ est la r\'egion fondamentale de la forme cyclotomique $\Phi_n$ et son aire est $A_{\Phi_n}$. 
 
 Pour $n\ge 3$, $\calO_n$ est sym\'etrique par rapport \`a la premi\`ere bissectrice et sym\'etrique par rapport au point $O$. De plus, si $n$ est divisible par $4$, $\calO_n$ est sym\'etrique par rapport aux axes de coordonn\'ees. Si $n$ est impair, $\calO_{2n}$ s'obtient \`a partir de $\calO_n$ par sym\'etrie par rapport \`a un des axes de coordonn\'ees. 
 
 Pour $n=4$, $\calO_4$ est le disque $x^2+y^2\le 1$ et
$$
A_{\Phi_4}= \int_{-\infty}^{+\infty} \frac { \rmd t}{1+t^2}
 =\pi.
 $$
Quand $p$ est un nombre premier impair on a 
$$
A_{\Phi_p}= \int_{-\infty}^{+\infty} \frac { \rmd t}{ (1+t+t^2+\cdots+t^{p-1})^{(p-1)/2}}\cdot 
 $$
Par exemple $\calO_3$ est l'int\'erieur de l'ellipse $x^2+xy+y^2=1$ et
 $$
 A_{\Phi_3}= \int_{-\infty}^{+\infty} \frac { \rmd t}{1+t+t^2}
 =\frac {2\pi}{\sqrt 3}\cdotp
$$

 \subsection{D\'emontration du th\'eor\`eme \ref{Thm:ovalecyclotomique}}

La d\'emonstration du th\'eor\`eme \ref{Thm:ovalecyclotomique} repose sur des estimations de $\Phi_n(t)$: majorations et minorations. L'estimation de $\rmm(\phi_n)$ donn\'ee dans \cite{FLW} ne suffit pas pour d\'emontrer le th\'eor\`eme \ref{Thm:ovalecyclotomique}. 

Montrer que $\calO_n$ contient le petit carr\'e revient \`a d\'emontrer, pour $n$ suffisamment grand, 
\begin{equation}\label{equation:contientlepetitcarre}
\Phi_n(x,y)<1
 \quad\hbox{ quand } \quad
\max\{|x|,|y|\}<1-n^{-1+\varepsilon},
\end{equation}
alors que montrer que $\calO_n$ est contenu dans le grand carr\'e revient \`a d\'emontrer, pour $n$ suffisamment grand, 
\begin{equation}\label{equation:contenudanslegrandcarre}
\Phi_n(x,y)>1
 \quad\hbox{ quand } \quad
\max\{|x|,|y|\}>1+n^{-1+\varepsilon}.
\end{equation}
Les relations $\Phi_n(x,y)=\Phi_n(y,x)=\Phi_n(-x,-y)$ (pour $n\ge 3$) permettent de se limiter au domaine $|y|\le x$.

Si $P\in\R[X]$ est un polyn\^ome de degr\'e $d$, on a 
$$
|P(t)|\le \rmL(P) \max\{1,|t|\}^d
$$
pour tout $t\in\R$. En particulier 
$$
\phi_n(t)\le \rmL(\phi_n) \max\{1,|t|\}^{\varphi(n)}
$$
pour tout $t\in\R$. 

De \eqref{equation:Bateman} on d\'eduit, pour tout $\varepsilon>0$, pour $n$ suffisamment grand et pour tout $t\in\R$, l'in\' egalit\' e
\begin{equation}\label{equation:majorationPhin}
\phi_n(t)\le e^{n^\varepsilon} \max\{1,|t|\}^{\varphi(n)}.
\end{equation}
Montrons que cela implique \eqref{equation:contientlepetitcarre}. 
Soit $n$ suffisamment grand et soit $(x,y)$ satisfaisant 
$$
0<|y|\le x<1-n^{-1+\varepsilon}.
$$
Posons $t=x/y$. On a $|t|\ge 1$ et, en utilisant \eqref{equation:majorationPhin} avec $\varepsilon/3$, 
$$
\Phi_n(x,y)=y^{\varphi(n)}\phi_n(t)\le y^{\varphi(n)}e^{n^{\varepsilon/3}} t^{\varphi(n)}
=x^{\varphi(n)} e^{n^{\varepsilon/3}} < \bigl(1-n^{-1+\varepsilon}\bigr)^{\varphi(n)} e^{n^{\varepsilon/3}}.
$$
Pour $n$ suffisamment grand on a 
$$
\varphi(n)>n^{1-\varepsilon/3}, \qquad
\log x\le \log \bigl(1-n^{-1+\varepsilon}\bigr)< - n^{-1+\varepsilon},
$$
d'o\`u
$$
\varphi(n)\log x< - n^{2\varepsilon/3},
$$
c'est-\`a-dire
$$
x^{\varphi(n)} < e^{-n^{2\varepsilon/3}}.
$$
Ceci compl\`ete la d\'emonstration de \eqref{equation:contientlepetitcarre}.

Pour d\'emontrer \eqref{equation:contenudanslegrandcarre}, on doit minorer $\phi_n(t)$. 
Nous utiliserons les estimations donn\'ees par les deux lemmes suivants; la premi\`ere nous sera utile quand $|t|-1$ n'est pas trop petit, la suivante quand $|t|-1$ est positif et petit.

\begin{lem}\label{lemme:minoratiodephin}
Pour tout $t\in\R$ et pour tout $n\geq 3$, on a 
$$
\phi_n(t)\ge |t|^{\varphi(n)-1}(|t|-1)
\prod_{\atop{d\mid n}{d\neq n}} \rmL(\phi_d)^{-1}.
$$
\end{lem}

\begin{proof}
Le r\'esultat est trivial si $|t|\le 1$. 
Pour commencer prenons $t>1$. De
\begin{equation}\label{equation:tn-1}
t^n-1=\phi_n(t)\prod_{\atop{d\mid n}{d\neq n}} \phi_d(t)
\end{equation}
on d\'eduit 
\begin{equation}\label{equation:majorationtn-1}
t^n-1\le \phi_n(t)\prod_{\atop{d\mid n}{d\neq n}}( \rmL(\phi_d) t^{\varphi(d)})=
\phi_n(t)t^{n-\varphi(n)}
\prod_{\atop{d\mid n}{d\neq n}} \rmL(\phi_d).
\end{equation}
 On minore $t^n-1$  par $ (t-1)t^{n-1}$.
Par cons\'equent, 
 $$
t^{\varphi(n)-1}(t-1)\le \phi_n(t)
\prod_{\atop{d\mid n}{d\neq n}} \rmL(\phi_d),
$$
ce qui est la conclusion du lemme \ref{lemme:minoratiodephin} pour $t>1$.

Supposons $t<-1$ et $n$ pair. Dans ce cas, $t^n-1=|t|^n-1$; dans le produit \eqref{equation:tn-1} il y a deux facteurs n\'egatifs, \`a savoir $\phi_1(t)=t-1$ et $\phi_2(t)=t+1$, et on remplace \eqref{equation:majorationtn-1} par
$$
|t|^n-1\le 
\phi_n(t)|t|^{n-\varphi(n)}
\prod_{\atop{d\mid n}{d\neq n}} \rmL(\phi_d).
$$
On conclut avec
$$
 |t|^n-1\ge (|t|-1)|t|^{n-1}.
 $$
Enfin pour $t<-1$ et $n$ impair, le membre de gauche de \eqref{equation:tn-1} est $-|t|^n-1$, et dans le membre de droite le seul facteur n\'egatif est celui correspondant \`a $d=1$, \`a savoir $\phi_1(t)=-|t|-1$. Ainsi
$$
|t|^n+1 =
 \phi_n(t)\prod_{\atop{d\mid n}{d\neq n}} |\phi_d(t)|
 \le 
\phi_n(t)|t|^{n-\varphi(n)}
\prod_{\atop{d\mid n}{d\neq n}} \rmL(\phi_d).
$$
Comme $|t|^n+1$ est minor\'e par $|t|^n$ on a une estimation plus pr\'ecise que celle du lemme \ref{lemme:minoratiodephin}, \`a savoir
 $$
\vert t \vert^{\varphi(n)} \le \phi_n(t)
\prod_{\atop{d\mid n}{d\neq n}} \rmL(\phi_d).
$$ 
\end{proof}
\begin{lem}\label{lemme:majoratiodephin-1}
Pour tout $t\in\R$ et pour tout $n\geq 1$, on a 
$$
|\phi_n(t)-\phi_n(1)|
\le |t-1| \max\{1,|t|\}^{\varphi(n)-1}\varphi(n)\rmL(\phi_n)
$$
et
$$ 
|\phi_n(t)-\phi_n(-1)|
\le |t+1| \max\{1,|t|\}^{\varphi(n)-1}\varphi(n)\rmL(\phi_n).
$$
\end{lem}

\begin{proof}
On pourrait faire intervenir la d\'eriv\'ee de $\phi_n$ dont la longueur est major\'ee par $\varphi(n)\rmL(\phi_n)$, mais on peut aussi faire un calcul direct comme ceci. 
\'Ecrivons
$$
\phi_n(t)=\sum_{j=0}^{\varphi(n)}a_jt^j.
$$
Alors $a_0+a_1+\cdots+a_{\varphi(n)}=\phi_n(1)$, $|a_0|+|a_1|+\cdots+|a_{\varphi(n)}|=\rmL(\phi_n)$ et 
$$
\phi_n(t)-\phi_n(1)=\sum_{j=1}^{\varphi(n)}a_j(t^j-1).
$$
On \'ecrit 
$$
\frac{t^j-1}{t-1}=\sum_{i=0}^{j-1} t^i
$$
et
$$
\frac{|t^j-1|}{|t-1|}\le j \max\{1,|t|\}^{j-1}\le \varphi(n) \max\{1,|t|\}^{\varphi(n)-1},
$$
ce qui donne 
$$
|\phi_n(t)-\phi_n(1)|\le |t-1| \max\{1,|t|\}^{\varphi(n)-1} \varphi(n)\sum_{j=1}^{\varphi(n)} |a_j|.
$$
La m\^eme d\'emonstration donne
$$
|\phi_n(t)-\phi_n(-1)|\le |t+1| \max\{1,|t|\}^{\varphi(n)-1} \varphi(n)\sum_{j=1}^{\varphi(n)} |a_j|.
$$
\end{proof}

\begin{proof}[D\'emonstration de \eqref{equation:contenudanslegrandcarre}]
Soit $n$ un entier suffisamment grand et soit $(x,y)\in\R^2$ v\'erifiant $0<|y|<x$ et
$$
x>1+n^{-1+\varepsilon}.
$$
On a 
$$
\log x> \frac 12 n^{-1+\varepsilon}
\quad\hbox{ et }\quad
\varphi(n)\log x>n^{2\varepsilon/3}.
$$
 On pose $t=x/y$, de sorte que $|t|>1$. On \'ecrit
$$
\Phi_n(x,y)=y^{\varphi(n)}\phi_n(t)
$$
et on minore $\phi_n(t)$ en consid\'erant deux cas.

\noindent
$\bullet$ \emph{Premier cas}. Supposons 
$$
|t|>1+e^{-n^{\varepsilon/2}}.
$$
Cette minoration implique 
$$
\frac{|t|-1} {|t|}>\frac 12 e^{-n^{\varepsilon/2}}.
$$
On utilise les lemmes \ref{lemme:minoratiodephin} et \eqref{equation:Bateman}, o\`u $\varepsilon$ est remplac\'e par $\varepsilon/2$, pour obtenir
$$
\phi_n(t)\ge |t|^{\varphi(n)-1}(|t|-1)
 e^{-n^{\varepsilon/2}}\ge
 \frac 12
 t^{\varphi(n)}
 e^{-2n^{\varepsilon/2}},
$$
d'o\`u
$$
\Phi_n(x,y)\ge \frac 12 x^{\varphi(n)}
 e^{-2n^{\varepsilon/2}}.
$$
On a 
$$
\varphi(n) \log x>n^{2\varepsilon/3}>2n^{\varepsilon/2}+\log 2,
$$
ce qui donne 
$\Phi_n(x,y)>1$
pour $n$ suffisamment grand.

\noindent
$\bullet$ \emph {Deuxi\`eme cas}. 
Supposons maintenant 
$$
1< |t|\le 1+e^{-n^{\varepsilon/2}}.
$$
On a 
$$
\frac{|t|-1} {|t|}\le |t|-1\le e^{-n^{\varepsilon/2}}
\text{ et } \log (|t|-1)-\log |t| \le -n^{\varepsilon/2}.
$$
On utilise les majorations
$$
\varphi(n)\log |t| <n(|t|-1)\le ne^{-n^{\varepsilon/2}}
$$
et
$$
\log (|t|-1) +(\varphi(n)-1)\log |t| + n^{\varepsilon/3}+\log 2
\le
-n^{\varepsilon/2}+ne^{-n^{\varepsilon/2}}+ n^{\varepsilon/3}+\log 2
<0
$$
pour $n$ suffisamment grand. 
Les lemmes \ref{lemme:majoratiodephin-1} et \eqref{equation:Bateman}, o\`u $\varepsilon$ est remplac\'e par $\varepsilon/3$, donnent
$$
|\phi_n(t)-\phi_n(1)| < \frac 12 \;\; \text {si } \;\; t>1, \quad |\phi_n(t)-\phi_n(-1)| < \frac 12 \;\; \text {si } \;\; t<-1, 
$$
ce qui implique $\phi_n(t)>1/2$. Alors
$$
\Phi_n(x,y) > \frac 12 y^{\varphi(n)} 
$$
avec $|y|=x/|t|$. On a 
$$
\varphi(n)\log |y|=\varphi(n)\log x-\varphi(n)\log |t|,
$$
$$
 \log |t|\le e^{-n^{\varepsilon/2}}, \quad \varphi(n)\log |t| \le n e^{-n^{\varepsilon/2}},
$$
d'o\`u
$$
\varphi(n)\log x-\varphi(n)\log |t|\ge n^{2\varepsilon/3} -n e^{-n^{\varepsilon/2}}>\log 2
$$
pour $n$ suffisamment grand,
ce qui implique $\Phi_n(x,y)>1$. 
\end{proof}

La preuve du th\'eor\`eme \ref{Thm:ovalecyclotomique} est compl\`ete.

\section{Preuve du th\' eor\`eme \ref{thmcentral3}}
On se place sous les hypoth\`eses de ce th\' eor\`eme. Soit $d\geq 4$ un totient tel que $d+2$ soit aussi un totient.
Soit $n_1 <n_2<\cdots <n_t$ la liste des entiers tels que 
$$
n_i \not\equiv 2 \bmod 4 \text{ et }\varphi (n_i) =d,$$
et un entier $m$ tel que $\varphi (m) =d+2$.
On part de la minoration 
$$
\mathcal A_d (N) \geq \Bigl\vert \mathcal B_{n_1} (N) \cup\cdots \cup \mathcal B_{n_t} (N) \cup \mathcal B_m (N)
\Bigr\vert
$$
o\`u on utilise la notation \eqref{defcalB}. Par le principe d'inclusion--exclusion on a la minoration 
\begin{equation}\label{Ad>BB}
\mathcal A_d (N) \geq \Bigl\vert \mathcal B_{n_1} (N) \cup\cdots \cup \mathcal B_{n_t} (N) \Bigr\vert + \Bigl\vert \mathcal C_{\boldsymbol n, m} (N) 
\Bigr\vert,
\end{equation}
o\`u $\mathcal C_{\boldsymbol n, m} (N) $ est l'ensemble compl\' ementaire  
$$
\mathcal C_{\boldsymbol n, m} (N) := \bigl\{ u\in \mathbb Z\ \mid \ u \in \mathcal B_m (N) \text{ et } u \not\in B_{n_1} (N) \cup\cdots \cup \mathcal B_{n_t} (N)
\bigr\}. 
$$
Le premier terme \`a droite de la minoration \eqref{Ad>BB} est minor\' e en combinant \eqref{Ad>} et \eqref{minorationAd}: 
$$
 \Bigl\vert \mathcal B_{n_1} (N) \cup\cdots \cup \mathcal B_{n_t} (N) \Bigr\vert 
 \ge 
 C_d N^\frac 2d - O (N^{\eta_d+\varepsilon}).
$$
Pour minorer le cardinal de $\mathcal C_{\boldsymbol n, m} (N)$, nous commen\c cons par exhiber un ensemble de couples d'entiers $(a,b)$ de densit\'e positive dont l'image par $\Phi_m$
appartient \`a $\mathcal C_{\boldsymbol n, m} (\infty)$. On a 
 \begin{lem}\label{confinementapplique} Soit $d$, $t$, $n_1$, \dots, $n_t$ et $m$ des entiers comme ci--dessus. Il existe alors un entier $D$ et des 
 classes de congruence $a_0$ et $b_0\bmod D$ tels que
 $$a \equiv a_0 \text{ et } b\equiv b_0\bmod D \Rightarrow \Phi_m (a, b) \not\in \Bigl( \mathcal B_{n_1} (\infty) \cup\cdots \cup \mathcal B_{n_t} (\infty)\Bigr).
 $$
 
 \end{lem}
 \begin{proof}[D\' emonstration du lemme \ref{confinementapplique}] \`A chaque entier $n_i$ on associe l'entier $\varpi (n_i)$ d\' efini comme suit:
 \vskip .2cm 
 \noindent $\bullet $ si $n_i$ n'est pas de la forme $2^h 3^k$, alors $\varpi (n_i) $ est le plus petit diviseur premier $\geq 5$ de $n_i$, 
 \vskip .2cm 
 \noindent $\bullet $ si $n_i$ est de la forme $2^h 3^k$ avec $h\geq 2$, alors $\varpi (n_i) =4,$ 
 \vskip .2cm
 \noindent $\bullet $ si $n_i$ est de la forme $ 3^k$ avec $k\geq 2$, alors $\varpi (n_i) =9$. 

On pose alors 
$$
D:={\mathrm {ppcm}}\{ \varpi (n_i)\}.
$$
Pour d\'efinir $a_0$ et $b_0\bmod D$, nous allons fixer leurs classes de congruence modulo chacun des $\varpi (n_i)$, avant d'appliquer le th\' eor\`eme chinois pour remonter en des classes modulo $D$ : 
\vskip .2cm
\noindent $\bullet$
Si $\varpi (n_i)$ est un nombre premier $ \geq 5$, alors $\varpi (n_i) $ divise $n_i$ et $\varphi (\varpi (n_i)) = \varpi (n_i)-1 $ divise $d$. On a alors pour $a\equiv 0 \bmod \varpi (n_i)$ et $b \equiv 2 \bmod \varpi (n_i)$ les congruences 
$$
\Phi_m (a, b) \equiv b^{d+2} \equiv 4 \not= 0, \, 1 \bmod \varpi (n_i);
$$
donc $\Phi_m (a, b)$ n'appartient pas \`a l'image de $\Phi_{n_i}$ d'apr\`es la proposition \ref{prop:confinmodp}. On fixe $a_0\equiv0$ et $b_0\equiv 2 \bmod \varpi (n_i)$.\vskip .2cm
\noindent $\bullet$ Si $\varpi (n_i) =4$, c'est que $n_i$ est de la forme $2^h 3^k$ avec $h\geq 2$. On remarque que $d$ est divisible par $4$ (rappelons que $\varphi(n_i)=d\ge 4$). Dans ce cas $d+2$ est congru \`a $2$ modulo $4$. Les seuls $m$ tels que $\varphi (m) =d+2$ et $m\not\equiv 2 \bmod 4$ sont de la forme 
$m=p^s$ avec $p\equiv 3 \bmod 4$ et $s\geq 1$. Par la formule \eqref{phin(1)=}, on a pour $a$ et $b$ congru \`a $1$ modulo $4$, $\Phi_m (a,b) \equiv \phi_m (1)
\equiv p \equiv 3 \bmod 4$, et $\Phi_m (a,b)$ n'est pas dans l'image de $\Phi_{n_i}$, d'apr\`es la proposition \ref{prop:confinmod4}. On fixe donc $a_0\equiv b_0 \equiv 1 \bmod 4$.
\vskip .2cm
\noindent $\bullet$ Si $\varpi(n_i) =9$, alors $n_i $ est de la forme $3^k$ avec $k\geq 2$. Par cons\'equent, $6 \mid \varphi (n_i) =d$.
 Soit $m$ tel que $\varphi (m) =d+2.$
Alors
$\Phi_m (0, b)= b^{d+2}.$ Donc si $3 \nmid b$, on a $\Phi_m(0,b) \equiv b^2 \bmod 9.$ Si on impose $a \equiv 0\bmod 9$ et $b \equiv 2 \bmod 9$, on voit que $\Phi_m (a,b) \equiv 4 \bmod 9$.
Ce n'est pas une valeur prise par $\Phi_{3^k}$, par la proposition \ref{prop:confinmod9}. On fixe donc $a_0\equiv 0\bmod 9$ et $b_0 \equiv 2 \bmod 9$.

Le lemme \ref{confinementapplique} en r\' esulte.
 \end{proof}
Soient $M \geq 2$ et $\mathcal E (M)$ l'ensemble des couples d'entiers $(a,b)$ tels que $\vert a \vert ,\, \vert b\vert \leq M$ et $a\equiv a_0$ et $b\equiv b_0 \bmod D$, avec les notations du lemme \ref{confinementapplique}. 
Il existe $c_0>0$ tel que 
\begin{equation}\label{inclusion}
\Phi_m (\mathcal E(c_0 N^\frac1{d+2}))\subset \mathcal C_{\boldsymbol n, m} (N).
\end{equation}

Notons $\tilde \rho (n)$ le nombre de solutions de l'\' equation $n=\Phi_m (a,b)$ avec $(a,b)\in \mathcal E(c_0 N^\frac1{d+2})$ . On a donc l'\' egalit\' e 
\begin{equation}\label{sumrho}
\sum_{n} \tilde \rho (n) = \vert \mathcal E(c_0 N^\frac1{d+2})\vert \sim (4c_0^2/D^2) N^\frac 2{d+2}.
\end{equation}
Pour appliquer l'in\' egalit\'e de Cauchy--Schwarz, on \' ecrit la partie gauche de l'\'equation \eqref{sumrho} comme
\begin{equation}\label{Cauchy}
\sum_{n} \tilde\rho (n) =\sum_{\tilde \rho (n)\geq 1} 1 \cdot \tilde \rho (n)\leq \vert \Phi_m (\mathcal E(c_0 N^\frac1{d+2}))\vert^\frac 12 \cdot 
\Bigl( \sum_n \rho^2 (n) \Bigr)^\frac 12,
\end{equation}
o\`u $\rho (n)$ est le nombre de solutions \`a l'\' equation $n =\Phi_m(a,b) $ avec $\vert a \vert,\, \vert b \vert \leq c_0 N^\frac 1{d+2}.$
D\' eveloppant le carr\' e, on voit que $\sum \rho^2 (n)$ est le nombre de points entiers de hauteur $\ll N^\frac 1{d+2}$ sur la surface de $\mathbb P^3 (\mathbb C)$ d\' efinie par
$$
\Phi_m (X_1, X_2) -\Phi_m (X_3,X_4) =0.
$$
Cette surface est lisse de degr\' e $\geq 3$. Elle contient donc $O(1)$ droites. Sur chacune de ces droites il y a $O( N^\frac 2{d+2})$ points de hauteur
 $\ll N^\frac 1{d+2}$. Pour compter les points entiers non situ\' es sur ces droites, on suit la m\^eme d\' emonstration que pour la
 proposition \ref{N1<<} (voir aussi \cite[Lemma 2.4]{SX}). Le nombre de ces points entiers est en $O(N^\vartheta)$ pour un certain $\vartheta <2/(d+2)$. Regroupant les deux contributions, on a donc la majoration
 $$
 \sum_n \rho^2 (n) \ll N^\frac 2{d+2}.
 $$
 Combinant \eqref{sumrho} et \eqref{Cauchy} on obtient la minoration
 $$
 \vert \Phi_m (\mathcal E(c_0 N^\frac1{d+2}))\vert \gg N^\frac 2{d+2}.
 $$
 Retournant \`a \eqref{inclusion} puis \`a \eqref{Ad>BB} on compl\`ete la preuve du th\' eor\`eme \ref{thmcentral3}. 
  
\vfill

\end{document}